\documentclass{article}

\usepackage{microtype}
\usepackage{graphicx}
\usepackage{subfigure}
\usepackage{booktabs} 

\usepackage{amsmath, amssymb, amsfonts, listings}

\usepackage[colorinlistoftodos,bordercolor=orange,backgroundcolor=orange!20,linecolor=orange,textsize=scriptsize]{todonotes}

\usepackage{hyperref}

\newcommand{\defeq}{\stackrel{\mathrm{def}}{=}} 
\newcommand{\eqdef}{\stackrel{\mathrm{def}}{=}} 
\newcommand{\diag}{\operatorname{diag}} 

\newcommand{\BlockDiag}{\operatorname{blockdiag}}
\newcommand{\Sam}{\hat{S}}
\def\dom{\mathop{\rm dom}}
\def\argmin{\mathop{\rm argmin}}

\def\R{\mathbb{R}}
\def\E{\mathbb{E}}
\def\P{\mathbb{P}}
\def\H{\mathcal{H}}
\def\Hf{\mathcal{H}_{\!F}}

\newcommand{\mA}{{\bf A}}
\newcommand{\mB}{{\bf B}}

\newcommand{\mD}{{\bf D}}
\newcommand{\mE}{{\bf E}}

\newcommand{\mG}{{\bf G}}

\newcommand{\mI}{{\bf I}}

\newcommand{\mM}{{\bf M}}

\newcommand{\mP}{{\bf P}}

\newcommand{\mU}{{\bf U}}

\newcommand{\mZ}{{\bf Z}}

\newenvironment{proof}{{\noindent \textbf{Proof:}}}{\hfill\rule{2mm}{2mm}\par}

\newtheorem{assumption}{Assumption}

\newtheorem{lemma}{Lemma}
\newtheorem{theorem}{Theorem}
\newtheorem{proposition}{Proposition}

\usepackage[accepted]{icml2018_modified}

\icmltitlerunning{Randomized Block Cubic Newton Method}

\begin{document}

\twocolumn[
\icmltitle{Randomized Block Cubic  Newton Method}




\begin{icmlauthorlist}
	\icmlauthor{Nikita Doikov}{hse}
	\icmlauthor{Peter Richt\'{a}rik}{kaust,ed,mipt}
\end{icmlauthorlist}

\icmlaffiliation{hse}{National Research University Higher School of Economics, Moscow, Russia}
\icmlaffiliation{kaust}{King Abdullah University of Science and Technology, Thuwal,  Saudi Arabia}
\icmlaffiliation{ed}{University of Edinburgh, Edinburgh,  United Kingdom}
\icmlaffiliation{mipt}{Moscow Institute of Physics and Technology, Dolgoprudny, Russia}

\icmlcorrespondingauthor{Nikita Doikov}{nikita.doikov@gmail.com}
\icmlcorrespondingauthor{Peter Richt\'{a}rik}{peter.richtarik@kaust.edu.sa, peter.richtarik@ed.ac.uk}

\icmlkeywords{randomized algorithms, convex optimization, block coordinate descent, Newton method, cubic regularization, empirical risk minimization}

\vskip 0.3in
]



\printAffiliationsAndNotice{}  

\begin{abstract}
	We study the problem of minimizing the sum of three convex functions---a differentiable, twice-differentiable and a non-smooth term---in a high dimensional setting. To this effect we propose and analyze  a randomized block  cubic Newton (RBCN) method, which in each iteration builds a model of the objective function formed as the sum of the {\em natural} models of its three components: a linear model with a quadratic regularizer for the differentiable term, a quadratic model with a cubic regularizer for the twice differentiable term, and perfect (proximal)  model for the nonsmooth term. Our method in each iteration minimizes the model over a random subset of  blocks of the search variable. RBCN is the first algorithm with these properties, generalizing several existing methods, matching the best known bounds in all special cases. We establish ${\cal O}(1/\epsilon)$, ${\cal O}(1/\sqrt{\epsilon})$ and ${\cal O}(\log (1/\epsilon))$ rates under different assumptions on the component functions. Lastly, we show numerically that our method outperforms the state of the art on a variety of machine learning problems, including cubically regularized least-squares, logistic regression with constraints, and Poisson regression.
\end{abstract}

\section{Introduction}

In this paper we develop an efficient randomized algorithm for solving an optimization problem of the form

\begin{equation}\label{eq:main-no-structure} \min_{x\in Q}  F(x) \eqdef g(x) +  \phi(x) + \psi(x) , \end{equation}
where $Q\subseteq \R^N$ is a closed convex set, and $g, \phi$ and $\psi$ are convex functions with different smoothness and structural properties. Our aim is to capitalize on these different properties in the design of our algorithm.  We assume that $g$ has Lipschitz gradient\footnote{Our assumption is bit more general than this; see Assumptions~\ref{AssumptionStrongConvexityF},~\ref{AssumptionSmoothnessA} for details.}, $\phi$ has Lipschitz Hessian, while  $\psi$ is allowed to be nonsmooth, albeit ``simple''.  

\subsection{Block Structure}
Moreover, we assume that the $N$ coordinates of $x$ are partitioned into $n$ blocks  of sizes $N_1,\dots, N_n$, with $\sum_i N_i = N$, and then write $x = (x_{(1)},\dots,x_{(n)})$,  where $x_{(i)} \in \R^{N_i}$. This block structure is typically dictated by the particular application considered. Once the block structure is fixed,  we further assume  that $\phi$ and $\psi$ are {\em block separable}. That is, 
$\phi(x) = \sum_{i=1}^n \phi_i(x_{(i)})$ and $\psi(x) =  \sum_{i=1}^n \psi_i(x_{(i)})$,  where $\phi_i$ are twice differentiable with Lipschitz Hessians, and $\psi_i$ are closed convex (and possibly nonsmooth) functions. 

Revealing this block structure, problem \eqref{eq:main-no-structure} takes the form
\begin{equation}\label{eq:main-structure} \min_{x\in Q}  F(x) \eqdef g(x) +  \sum_{i=1}^n \phi_i(x_{(i)}) +\sum_{i=1}^n   \psi_i(x_{(i)}) . \end{equation}
We are specifically interested in the case when $n$ is {\em big}, in which case it make sense to update a small number of the block in each iteration only. 

\subsection{Related Work}

There has been a substantial and growing volume of research related to second-order and block-coordinate optimization. In this part we  briefly mention some of   the  papers most relevant to the present work. 

A major leap  in  second-order optimization theory was made since the cubic Newton method was proposed by \citet{Griewank1981}  and  independently rediscovered by \citet{nesterov2006cubic}, who  also  provided  global complexity guarantees. 

Cubic regularization was equipped with  acceleration by \citet{nesterov2008accelerating}, adaptive stepsizes by \cite{cartis2011adaptive1, cartis2011adaptive2} and extended to a universal framework by \citet{grapiglia2017regularized}. The universal schemes can automatically adjust to the implicit smoothness level of the objective. Cubically regularized second-order schemes for solving systems of nonlinear equations were developed by \citet{nesterov2007modified} and randomized variants for stochastic optimization were considered by \citet{tripuraneni2017stochastic, ghadimi2017second, kohler2017sub, cartis2018global}.

Despite their attractive global iteration complexity guarantees, the weakness of second-order methods in general, and cubic Newton in particular,  is their high computational cost per iteration. This issue  remains the subject of active research. For successful theoretical results related to the approximation of the cubic step we refer to~\cite{agarwal2016finding} and~\cite{carmon2016gradient}.

At the same time, there are many successful attempts to use \textit{block coordinate} randomization to accelerate first-order~\cite{tseng2009coordinate, richtarik2014iteration, richtarik2016parallel} and second-order~\cite{qu2016sdna, mutny2018parallel} methods. 

In this work we are addressing the issue of combining block-coordinate randomization with cubic regularization, to get \textit{a second-order method with proven global complexity guarantees and with a low cost per iteration}.

A  powerful advance in convex optimization theory was the advent of  \textit{composite} or \textit{proximal} first-order methods (see \cite{nesterov2013gradient} as a modern reference). This technique has become available as an algorithmic tool in block coordinate setting as well~\cite{richtarik2014iteration, qu2016sdna}.  Our aim in this work is the development of a \textit{composite cubically regularized second-order method}.

\subsection{Contributions}

We propose a new randomized second-order proximal algorithm for solving convex optimization problems of the form~\eqref{eq:main-structure}. Our method, {\em Randomized Block Cubic Newton (RBCN)} (see Algorithm~\ref{MainAlgorithm}) treats the three functions appearing in \eqref{eq:main-no-structure} differently, according to their nature. 

Our method is a {\em randomized block method} because in each iteration we update a random subset of the $n$ blocks only. This facilitates faster convergence, and is suited to problems where $n$ is very large. Our method is {\em proximal} because we keep the functions $\psi_i$ in our model, which is minimized in each iteration, without any approximation. Our method is a {\em cubic Newton} method because we approximate each $\phi_i$ using a cubically-regularized second order model.  

We are not aware of {\em any method} that can solve \eqref{eq:main-structure} via using the most appropriate models of the three functions (quadratic with a constant Hessian for $g$, cubically regularized quadratic for $\phi$ and no model for $\psi$), not even in the case $n=1$.

Our approach generalizes several existing results:

\begin{itemize}
	\item In the case when $n=1$, $g=0$ and $\psi=0$, RBCN reduces to the cubically-regularized Newton method of \citet{nesterov2006cubic}. Even when  $n=1$,  RBCN can be seen as an extension of this method to  \textit{composite} optimization. For $n>1$, RBCN provides an extension of the algorithm in \citet{nesterov2006cubic} to the  \textit{randomized block coordinate} setting, popular for high-dimensional problems.
	
	\item In the special case when $\phi=0$ and $N_i=1$ for all $i$, RBCN specializes to  the stochastic Newton (SN) method of \citet{qu2016sdna}. Applied to the empirical risk minimization problem (see Section~\ref{sec:ERM}), our method has a dual interpretation (see Algorithm~\ref{DualAlgorithm}). In this case, our method reduces to the stochastic dual Newton ascent method (SDNA) also described in \cite{qu2016sdna}. Hence, RBCN can be seen as an extension of SN and SDNA to blocks of arbitrary sizes, and to the inclusion of the twice differentiable term $\phi$.

	\item  In the case when $\phi = 0$ and the simplest over approximation of $g$ is assumed: $0 \preceq \nabla^2 g(x) \preceq LI$, the composite block coordinate gradient method \citet{tseng2009coordinate} can be applied to solve~\eqref{eq:main-no-structure}. Our method extends this in two directions: we add twice-differentiable terms~$\phi$, and use a tighter model for $g$, using all global curvature information (if available).

\end{itemize}

We prove  high probability global convergence guarantees under several regimes, summarized next:

\begin{itemize}
	\item Under no additional assumptions on $g$, $\phi$ and $\psi$ beyond convexity (and either boundedness of $Q$, or boundedness of the level sets of $F$ on $Q$), we prove the rate \[{\cal O}\left(\frac{n}{\tau \epsilon}\right),\] where $\tau$ is the mini-batch size (see Theorem~\ref{TheoremConvexSublinear}).
	\item Under certain conditions combining the properties of $g$ with  the way the random blocks are sampled, formalized by the assumption $\beta>0$ (see \eqref{BetaDefinition} for the definition of $\beta$), we obtain the rate \[{\cal O}\left(\frac{n}{\tau  \max\{1,\beta\} \sqrt{\epsilon} }\right)\] (see Theorem~\ref{TheoremStronglyConvexBeta}). In the special case when $n=1$, we necessarily have $\tau=1$ and  $\beta=\mu/L$  (reciprocal of the condition number of $g$) we get the rate ${\cal O}(\tfrac{L}{\mu\sqrt{\epsilon}})$.
	If $g$ is quadratic and $\tau=n$, then $\beta=1$ and the resulting complexity ${\cal O}(\tfrac{1}{\sqrt{\epsilon}})$ recovers the  rate of cubic Newton established by \citet{nesterov2006cubic}.
	\item Finally, if $g$ is strongly convex, the above result can be improved (see Theorem~\ref{TheoremStronglyConvexMu}) to
	\[{\cal O}\left( \frac{n}{\tau \max\{1,\beta\}} \log \frac{1}{\epsilon}\right).\]

\end{itemize}

\subsection{Contents}

The rest of the paper is organized as follows. In Section~\ref{sec:block} we introduce the notation and elementary identities needed to efficiently handle the block structure of our model. In Section~\ref{sec:assumptions} we make the various smoothness and convexity assumptions on $g$ and $\phi_i$ formal. Section~\ref{sec:samplings} is devoted to the description of the block sampling
process used in our method, along with some useful identities. In Section~\ref{sec:algorithm} we describe formally our randomized block cubic Newton (RBCN) method. Section~\ref{sec:convergence} is devoted to the statement and description of our main convergence results, summarized in the introduction. Missing proofs are provided in the supplementary material. In Section~\ref{sec:ERM} we show how to apply our method to the empirical risk minimization problem. Applying RBCN to its dual leads to Algorithm~\ref{DualAlgorithm}. Finally, our numerical experiments on synthetic and real datasets are described in Section~\ref{sec:numerics}.

\section{Block Structure} \label{sec:block}

To model a block structure, we decompose the space $\R^N$ into $n$ subspaces in the following standard way.
Let $\mU \in \R^{N \times N}$ be a column permutation of the $N \times N$ identity matrix $\mI$ and let a decomposition $\mU = [\mU_1, \mU_2, \dots, \mU_n]$ be given, where $ \mU_i \in \R^{N \times N_i}$ are $n$ submatrices, $N = \sum_{i = 1}^n N_i$. Subsequently, any vector $x \in \R^N$ can be uniquely represented as $x = \sum_{i = 1}^n \mU_i x_{(i)}$, where $x_{(i)} \defeq \mU_i^T x \in \R^{N_i}$.

In what follows we will use the standard Euclidean inner product: $\langle x, y \rangle \defeq \sum_i x_i y_i$, Euclidean norm of a vector: $\|x\| \defeq \sqrt{\langle x, x \rangle}$ and induced spectral norm of a matrix: $\|\mA\| \defeq \max_{\|x\|=1}\|\mA x\|$. Using block decomposition, for two vectors $x, y \in \R^N$ we have:
$$
\langle x, y \rangle = \biggl\langle \sum_{i = 1}^n \mU_i x_{(i)}, \sum_{j = 1}^n \mU_j y_{(j)} \biggl\rangle = \sum_{i = 1}^n \langle x_{(i)}, y_{(i)} \rangle.
$$
For a given nonempty subset $S$ of $[n] \defeq \{1, \dots, n\}$ and for any vector $x \in \R^{N}$ we denote by $x_{[S]} \in \R^N$ the vector obtained from $x$ by retaining only blocks $x_{(i)}$ for which $i \in S$ and zeroing all other:
$$
x_{[S]} \defeq \sum_{i \in S} \mU_i x_{(i)} = \sum_{i \in S} \mU_i \mU_i^T x.
$$
Furthermore, for any matrix $\mA \in \R^{N \times N}$ we write $\mA_{[S]} \in \R^{N \times N}$ for the matrix obtained from $\mA$ by retaining only elements whose indices are both in some coordinate blocks from $S$, formally:
$$
\mA_{[S]} \defeq \Biggl(\sum_{i \in S} \mU_i \mU_i^T\Biggr) \mA \Biggl(\sum_{i \in S} \mU_i \mU_i^T\Biggr).
$$
Note that  these definitions imply that
$$
\langle \mA_{[S]}x, y \rangle = \langle \mA x_{[S]}, y_{[S]} \rangle, \quad  x, y \in \R^N.
$$
Next, we define the \textit{block-diagonal} operator, which, up to permutation of coordinates, retains  diagonal blocks and nullifies the off-diagonal blocks:
$$
\BlockDiag(\mA) \defeq \sum_{i = 1}^n \mU_i \mU_i^T \mA \mU_i \mU_i^T = \sum_{i = 1}^n \mA_{[\{i\}]}.
$$
Finally, denote  $\R^N_{[S]} \defeq \bigl\{ x_{[S]} \, | \, x \in \R^N \bigr\} $. This is a linear subspace of $\R^N$ composed of vectors which are zero in blocks $i\notin S$.

\section{Assumptions} \label{sec:assumptions}

In this section we  formulate our main assumptions about differentiable components of~\eqref{eq:main-structure} and provide some examples to illustrate the concepts.

We assume that $g: \R^N \to \R$ is a differentiable function and all $\phi_i: \R^{N_i} \to \R, \; i \in [n]$ are twice differentiable. Thus, at any point $x \in \R^N$ we should be able to compute all the gradients $\{ \nabla g(x), \nabla \phi_1 (x_{(1)}), \dots, \nabla \phi_n (x_{(n)}) \}$ and the Hessians $\{ \nabla^2 \phi_1 (x_{(1)}), \dots, \nabla^2 \phi_n(x_{(n)}) \}$, or at least their actions on arbitrary vector $h$ of appropriate dimension.

Next, we formalize our assumptions about convexity and level of smoothness. Speaking informally,  $g$ is similar to a quadratic, and functions  $\phi_i$ are arbitrary twice-differentiable and smooth.

\begin{assumption}[Convexity] \label{AssumptionStrongConvexityF}
	There is a positive semidefinite matrix $\mG \succeq 0$ such that for all $x, h \in \R^N$:
	\begin{align} \label{StrongConvexityF}
	g(x + h) &\geq g(x) + \langle \nabla g(x), h \rangle + \frac{1}{2}\langle \mG h, h \rangle, \\[3pt]
	\phi_i(x_{(i)} + h_{(i)}) &\geq \phi_i(x_{(i)}) + \langle \nabla \phi_i(x_{(i)}), h_{(i)} \rangle, \quad i \in [n]. \notag
	\end{align}
\end{assumption}

In the special case when  $\mG = 0$, \eqref{StrongConvexityF} postulates \textit{convexity}. For positive definite $\mG$, the objective will be \textit{strongly convex} with the strong convexity parameter $\mu \defeq \lambda_{\min}(\mG) > 0$.

Note that for all $\phi_i$ we  only require convexity. However, if  we happen to know that any  $\phi_i$ is strongly convex ($\lambda_{\min}(\nabla^2 \phi_i( y ) ) \geq \mu_i > 0$ for all $y \in \R^{N_i}$), we can \textit{move} this strong convexity to $g$ by subtracting $ \frac{\mu_i}{2}\|x_{(i)}\|^2$ from $\phi_i$ and adding it to $g$. This extra knowledge may in some particular cases improve convergence guarantees for our algorithm, but does not change the actual computations.

\begin{assumption}[Smoothness of $\boldsymbol{g}$] \label{AssumptionSmoothnessA} There is a positive semidefinite matrix $\mA \succeq 0$
	such that for all $x, h \in \R^N$:
	\begin{equation} \label{SmoothnessA}
	g(x + h) \leq g(x) + \langle \nabla g(x), h \rangle + \frac{1}{2} \langle \mA h, h \rangle.
	\end{equation}
\end{assumption}

The main example of $g$ is a quadratic function $g(x) = \frac{1}{2}\langle \mM x, x \rangle$ with a symmetric positive semidefinite $\mM \in \R^{N \times N}$ for which both~\eqref{StrongConvexityF} and~\eqref{SmoothnessA} hold with $\mG = \mA = \mM$. 

Of course, any convex $g$ with Lipschitz-continuous gradient with a constant $L \geq 0$ satisfies~\eqref{StrongConvexityF} and~\eqref{SmoothnessA} with $\mG = 0$ and $\mA = L \mI$~\cite{nesterov2004introductory}.

\begin{assumption}[Smoothness of $\boldsymbol{\phi_i}$] \label{AssumptionSmoothnessPhi} For every $i \in [n]$ there is a nonnegative constant $\H_i \geq 0$ such that the Hessian of $\phi_i$ is Lipschitz-continuous:
	\begin{equation} \label{SmoothnessPhi}
	\| \nabla^2 \phi_i(x + h) - \nabla^2 \phi_i(x) \| \leq \H_i \|h\|,
	\end{equation}
	for all $x, h \in \R^{N_i}$.
\end{assumption}

Examples of functions which satisfy~\eqref{SmoothnessPhi} with a known Lipschitz constant of Hessian $\H$ are \textit{quadratic}: $\phi(t) = \|Ct - t_0\|^2$ ($\H = 0$ for all the parameters), \textit{cubed norm}: $\phi(t) = (1/3)\|t - t_0\|^3$ ($\H = 2$, see Lemma~5 in \cite{nesterov2008accelerating}), \textit{logistic regression loss}: $\phi(t) = \log(1 + \exp(t))$ ($\H = 1 / (6 \sqrt{3})$, see Proposition~1 in the supplementary material).

For a fixed set of indices $S \subset [n]$ denote
$$\phi_S(x) \defeq \sum_{i \in S} \phi_i(x_{(i)}), \quad x \in \R^N.$$
Then we have:
\begin{align*}
\langle \nabla \phi_S(x), h \rangle &=  \sum_{i \in S} \langle \nabla \phi_i(x_{(i)}), h_{(i)} \rangle, \quad x, h \in \R^N, \\
\langle \nabla^2 \phi_S(x)h, h \rangle &=\sum_{i \in S} \langle \nabla^2 \phi_i(x_{(i)}) h_{(i)}, h_{(i)} \rangle, \quad x, h \in \R^N.
\end{align*}

\begin{lemma} \label{Lemma:LemmaPhiBound} 
	If  Assumption~\ref{AssumptionSmoothnessPhi} holds, then for all $x, h \in \R^N$ we have the following second-order approximation bound:
	\begin{align} \label{LemmaPhiBound} \notag
	\Bigl| \phi_S(x + h) - \phi_S(x) &- \langle \nabla \phi_S(x), h \rangle - \frac{1}{2}\langle \nabla^2 \phi_S(x)h, h \rangle  \Bigr| \\
	\leq& \quad \max_{i \in S}\{ \H_i \} \cdot \|h_{[S]}\|^3. 
	\end{align}
\end{lemma}

From now on we denote $\displaystyle \Hf \defeq \max \{\H_1, \H_2, \dots, \H_n\}$. 

\section{Sampling of Blocks} \label{sec:samplings}

In this section we summarize some basic properties of  sampling $\hat{S}$, which is a random set-valued mapping with values being subsets of $[n]$.  For a fixed block-decomposition, with each sampling $\hat{S}$ we associate the {\em probability matrix} $\mP \in \R^{N \times N}$ as follows: an element of $\mP$ is the probability of choosing a pair of blocks which contains indices of this element. Denoting by $\mE \in \R^{N \times N}$ the matrix of all ones, we have $\mP = \E\bigl[ \mE_{[\hat{S}]} \bigr]$. Wel restrict our analysis to  {\em uniform} samplings, defined next.
\begin{assumption}[Uniform sampling] \label{AssumptionUniformSampling}
	Sampling $\hat{S}$ is uniform, i.e.,
	$\P(i \in \hat{S}) = \P(j \in \hat{S}) \eqdef p$, for all $i ,j \in [n]$.
\end{assumption}
The above assumpotion means that the diagonal of $\mP$ is constant: $\mP_{ii} = p$ for all $i \in [N]$. It is easy to see that  (Corollary 3.1 in \cite{ESO}):
\begin{equation} \label{eq:h9s8y8yhs}\E\bigl[ \mA_{[\hat{S}]}  \bigr] = \mA\circ \mP,\end{equation}
where $\circ$ denotes the Hadamard product. 

Denote $\tau \defeq \E[|\hat{S}|] = np$ (expected minibatch size).  The special uniform sampling defined by picking from all subsets of size $\tau$ uniformly at random is called \textit{$\tau$-nice sampling}. If $\hat{S}$ is $\tau$-nice, then  (see Lemma 4.3 in \cite{ESO}) \begin{equation}\label{eq:b9d898gdy8h9} \mP = \frac{\tau}{n}\left((1-\gamma)\BlockDiag(\mE) + \gamma \mE\right),\end{equation}
where $\gamma = (\tau-1)/(n-1)$. 

In particular, the above results in the following:
\begin{lemma} For the $\tau$-nice sampling $\hat{S}$, we have
	$$\E[ \mA_{[\hat{S}]}  ] = \frac{\tau}{n}\left(1 - \frac{\tau-1}{n-1}\right)\BlockDiag(\mA) + \frac{\tau(\tau-1)}{n(n-1)} \mA . $$
\end{lemma}
\begin{proof} Combine \eqref{eq:h9s8y8yhs} and \eqref{eq:b9d898gdy8h9}.
\end{proof}

\section{Algorithm} \label{sec:algorithm}

Due to the problem structure~\eqref{eq:main-structure} and utilizing the smoothness of the components (see \eqref{SmoothnessA} and \eqref{SmoothnessPhi}), for a fixed subset of indices $S \subset [n]$ it is natural to consider the following \textit{model} of our objective $F$ around a point $x \in \R^N$:
\begin{align}  
&M_{H, S}(x; y) \defeq F(x) + \langle (\nabla g(x))_{[S]}, y \rangle + \frac{1}{2}\langle \mA_{[S]} y, y \rangle + \notag \\ 
& +  \langle ( \nabla\phi(x))_{[S]}, y \rangle  + \frac{1}{2}\langle (\nabla^2 \phi(x))_{[S]} y, y \rangle + \frac{H}{6}\|y_{[S]}\|^3 + \notag \\
&+ \sum_{i \in S} \Bigl(  \psi_i(x_{(i)} + y_{(i)}) - \psi_i(x_{(i)})\Bigr). \label{CubicModelMain}
\end{align}
The above  model arises as a combination of a first-order model for $g$ with global curvature information provided by matrix $\mA$, second-order model with cubic regularization (following \cite{nesterov2006cubic}) for  $\phi$, and perfect model  for the non-differentiable terms $\psi_i$ (i.e., we keep these terms as they are).

Combining~\eqref{SmoothnessA} and~\eqref{LemmaPhiBound}, and  for large enough $H$ ($ H \geq \max_{i \in S} \H_i$ is sufficient), we get the global upper bound 
$$
F(x + y) \leq M_{H, S}(x; y), \quad x \in \R^N, \; y \in \R_{[S]}^N.
$$
Moreover, the value of all summands in $M_{H,S}(x; y)$ depends  on the subset of blocks $\{ y_{(i)} | i \in S \}$ only, and therefore 
\begin{equation} \label{DefinitionOfStep}
T_{H, S}(x) \defeq \argmin_{\substack{y \in \R_{[S]}^N \\[2pt] \text{subject to} \; x+y \in Q}} M_{H, S}(x; y)
\end{equation}
can be computed efficiently for small $|S|$ and as long as $Q$ is \textit{simple} (for example, affine). Denote the minimum of the cubic model by $M^{*}_{H, S}(x) \defeq M_{H,S}(x; T_{H,S}(x))$. The RBCN method performs the update $x \leftarrow x + T_{H, S}(x)$, and is  formalized as Algorithm~\ref{MainAlgorithm}.

\begin{algorithm}[ht!]
	\caption{RBCN: Randomized Block  Cubic Newton}
	\label{MainAlgorithm}
	\begin{algorithmic}[1]
		\STATE \textbf{Parameters:} sampling distribution $\Sam$
		\STATE \textbf{Initialization:} choose initial point $x^0 \in Q$
		\FOR {$k=0,1,2,\dots$}
		\STATE Sample $S_k \sim \Sam$
		\STATE Find $H_k \in (0,  \, 2 \Hf]$ such that \\$\qquad F(x^k + T_{H_k, S_k}(x^k)) \leq M^{*}_{H_k, S_k}(x^k)$
		\STATE Make the step $x^{k + 1} \eqdef x^k + T_{H_k, S_k}(x^k)$
		\ENDFOR
	\end{algorithmic}
\end{algorithm}

\section{Convergence Results} \label{sec:convergence}

In this section we establish several convergence rates for  Algorithm~\ref{MainAlgorithm} under various structural assumptions:  for the general class of convex problems, and for the more specific strongly convex case. We will focus on the family of \textit{uniform samplings} only, but generalizations to other sampling distributions are also possible.

\subsection{Convex Loss} We start from the general situation where the term $g(x)$ and all the $\phi_i(x_{(i)})$ and $\psi_i(x_{(i)}), i \in [n]$ are convex, but not necessary strongly convex.

Denote by $D$ the maximum distance from an optimum point $x^{*}$ to the initial level set:
$$D \defeq \sup \Bigl\{ \|x - x^{*}\| \; \big| \; x \in Q, \,  F(x) \leq F(x^0) \Bigr\}.$$

\begin{theorem} \label{TheoremConvexSublinear} Let Assumptions~\ref{AssumptionStrongConvexityF},~\ref{AssumptionSmoothnessA},~\ref{AssumptionSmoothnessPhi},~\ref{AssumptionUniformSampling} hold. Let solution $x^{*} \in Q$ of problem~\eqref{eq:main-no-structure} exist, and assume the level sets are bounded: $D < +\infty$. Choose required accuracy $\varepsilon > 0$ and confidence level $\rho \in (0, 1)$. Then after
	\begin{equation} \label{Theorem:ConvexSublinearRate}
	K \geq \frac{2}{\varepsilon} \frac{n}{\tau} \biggl(1 + \log\frac{1}{\rho} \biggr) \max\Bigl\{ LD^2 + \Hf D^3, \, F(x^{0}) - F^{*} \Bigr\}  
	\end{equation}
	iterations of Algorithm~\ref{MainAlgorithm}, where $L \defeq \lambda_{\max}(\mA)$, we have
	$$
	\P\Bigl( F(x^K) - F^{*} \leq \varepsilon \Bigr) \; \geq \; 1 - \rho.
	$$
\end{theorem}

\medskip

Given theoretical result provides global sublinear rate of convergence, with iteration complexity of the order $O\bigl( 1 / \varepsilon \bigr)$. 

Note that for a case $\phi(x) \equiv 0$ we can put $\Hf = 0$, and Theorem~\ref{TheoremConvexSublinear} in this situation restates well-known result about convergence of composite gradient-type block-coordinate methods (see, for example, \cite{richtarik2014iteration}).

\subsection{Strongly Convex Loss} Here we study the case when the matrix $\mG$ from the convexity assumption~\eqref{StrongConvexityF} is strictly positive definite: $\mG \succ 0$, which means that the objective $F$ is \textit{strongly convex} with a constant $\mu~\defeq~\lambda_{\min}(\mG)~>~0$. 

Denote by $\beta$ a \textit{condition number} for the function $g$ and sampling distribution $\hat{S}$: the maximum nonnegative real number such that
\begin{equation} \label{BetaDefinition}
\beta \cdot \E_{S \sim \hat{S}}\bigl[ \mA_{[S]} \bigr] \; \preceq \; \frac{\tau}{n} \mG .
\end{equation}
If~\eqref{BetaDefinition} holds for \textit{all} nonnegative $\beta$ we put by definition  $\beta \equiv +\infty$. 

A simple lower bound exists: $\beta \geq \frac{\mu}{L} > 0$, where $L = \lambda_{\max}(\mA)$, as in Theorem~\ref{TheoremConvexSublinear}.
However, because~\eqref{BetaDefinition} depends not only on $g$, but also on sampling distribution $\hat{S}$, it is possible that $\beta > 0$ even if $\mu = 0$ (for example, $\beta = 1$ if $\P(S = [n]) = 1$ and $\mA = \mG \not= 0$).

The following theorems describe global iteration complexity guarantees of the order $O\bigl(1 / \sqrt{\varepsilon}\bigr)$ and $O\bigl( \log(1 / \varepsilon) \bigr)$ for Algorithm~\ref{MainAlgorithm} in the cases $\beta > 0$ and $\mu > 0$ correspondingly, which is an improvement of general $O\bigl(1 / \varepsilon\bigr)$.

\begin{theorem} \label{TheoremStronglyConvexBeta}
	Let Assumptions~\ref{AssumptionStrongConvexityF},~\ref{AssumptionSmoothnessA},~\ref{AssumptionSmoothnessPhi},~\ref{AssumptionUniformSampling} hold. Let solution $x^{*} \in Q$ of problem~\eqref{eq:main-no-structure} exist, let level sets be bounded: $D < +\infty$, and 
	assume that $\beta$, which is defined by~\eqref{BetaDefinition}, is greater than zero. 
	Choose required accuracy $\varepsilon > 0$ and confidence level $\rho \in (0, 1)$. Then after
	\begin{equation}
	K \geq \frac{2}{\sqrt{\varepsilon}} \frac{n}{\tau} \frac{1}{\sigma} \biggl(2 + \log\frac{1}{\rho} \biggr)   \sqrt{\max\Bigl\{ \Hf D^3, \, F(x^{0}) - F^{*} \Bigr\}} \label{Theorem:StronglyConvexSublinearRate}
	\end{equation}
	iterations of Algorithm~\ref{MainAlgorithm}, where $\sigma \defeq \min\{\beta, 1\} > 0$, we have
	$$
	\P\Bigl( F(x^K) - F^{*} \leq \varepsilon \Bigr) \; \geq \; 1 - \rho.
	$$
\end{theorem}

\begin{theorem} \label{TheoremStronglyConvexMu}
	Let Assumptions~\ref{AssumptionStrongConvexityF},~\ref{AssumptionSmoothnessA},~\ref{AssumptionSmoothnessPhi},~\ref{AssumptionUniformSampling} hold. Let solution $x^{*} \in Q$ of problem~\eqref{eq:main-no-structure} exist, and assume that $\mu \defeq \lambda_{\min}(\mG)$ is strictly positive. Then after
	\begin{equation} \label{Theorem:StronglyConvexLinearRate}
	K \geq \frac{3}{2}\log\biggl( \frac{F(x^0) - F^{*}}{\varepsilon \rho} \biggr) \frac{n}{\tau} \frac{1}{\sigma} \sqrt{\max\biggl\{ \frac{\Hf D}{\mu}, 1 \biggr\}} 
	\end{equation}
	iterations of Algorithm~\ref{MainAlgorithm}, we have
	$$
	\P\Bigl( F(x^K) - F^{*} \leq \varepsilon \Bigr) \; \geq \; 1 - \rho.
	$$
\end{theorem}

Given complexity estimates show which parameters of the problem directly affect on the convergence of the algorithm.

Bound~\eqref{Theorem:StronglyConvexSublinearRate} improves initial estimate~\eqref{Theorem:ConvexSublinearRate}  by the factor~$\sqrt{D_0 / \varepsilon}$. The cost is additional term $\sigma^{-1} = (\min\{\beta, 1\})^{-1}$, which grows up while the condition number~$\beta$ becomes smaller.  

Opposite and limit case is when the \textit{quadratic} part of the objective is vanished ($g(x) \equiv 0 \Rightarrow \sigma = 1$). Algorithm~\ref{MainAlgorithm} is turned to be a parallelized block-independent minimization of the objective components  via cubically regularized Newton steps. Then, the complexity estimate coincides with a known result~\cite{nesterov2006cubic} in a nonrandomized ($\tau = n, \, \rho \to 1$) setting.

Bound~\eqref{Theorem:StronglyConvexLinearRate} guarantees a linear rate of convergence, which means logarithmic dependence on required accuracy~$\varepsilon$ for the number of iterations. The main complexity factor becomes a product of two terms: $\sigma^{-1} \cdot \max\{ \H_F D / \mu, 1 \}^{1/2}$. For the case $\phi(x) \equiv 0$ we can put $\H_F = 0$ and get the stochastic Newton method  from~\cite{qu2016sdna} with its global linear convergence guarantee.

Despite the fact that linear rate is asymptotically better than sublinear, and $O\bigr(1/\sqrt{\varepsilon}\bigr)$ is asymptotically better than $O\bigl(1/\varepsilon\bigr)$, we need to take into account all the factors, which may slow down the algorithm. Thus, while $\mu = \lambda_{\min}(\mG) \to 0$, estimate~\eqref{Theorem:StronglyConvexSublinearRate} is becoming better than~\eqref{Theorem:StronglyConvexLinearRate}, as well as~\eqref{Theorem:ConvexSublinearRate} is becoming better than~\eqref{Theorem:StronglyConvexSublinearRate} while $\beta \to 0$.

\subsection{Implementation Issues}

Let us explain how one step of the method can be performed, which requires the minimization of the cubic model~\eqref{DefinitionOfStep}, possibly with some simple convex constraints. 

The first and the classical approach was proposed in~\cite{nesterov2006cubic} and, before for trust-region methods, in~\cite{conn2000trust}. It works with unconstrained ($Q \equiv \R^N$) and differentiable case ($\psi(x) \equiv 0)$. Firstly we need to find a root of a special one-dimensional nonlinear equation (this can be done, for example, by simple Newton iterations). 
After that, we just solve one linear system to produce a step of the method. Then, total complexity of solving the subproblem can be estimated as $O(d^3)$ arithmetical operations, where $d$ is the dimension of subproblem, in our case: $d = |S|$. Since some matrix factorization is used, the cost of the cubically regularized Newton step is actually similar by efficiency to the classical Newton one. See also~\cite{gould2010solving} for detailed analysis. For the case of affine constraints, the same procedure can be applied. Example of using this technique is given by Lemma~\ref{lem:b897sg99} from the next section.

Another approach is based on finding an inexact solution of the subproblem by the fast approximate eigenvalue computation~\cite{agarwal2016finding} or by applying gradient descent~\cite{carmon2016gradient}. Both of these schemes provide global convergence guarantees. Additionally, they are Hessian-free. Thus we need only a procedure of multiplying quadratic part of~\eqref{CubicModelMain} to arbitrary vector, without storing the full matrix. The latter approach is the most universal one and can be spread to the composite case, by using proximal gradient method or its accelerated variant~\cite{nesterov2013gradient}.

There are basically two strategies to find parameter $H_k$ on every iteration: a \textit{constant choice} $H_k := \max_{i \in S_k}\{ \H_i \}$ or $H_k := \H_F$, if Lipschitz constants of the Hessians are known, or simple \textit{adaptive procedure}, which performs a truncated binary search and has a logarithmic cost per one step of the method. Example of such procedure can be found in primal-dual Algorithm~\ref{DualAlgorithm} from the next section.

\subsection{Extension of the Problem Class}

The randomized cubic model~\eqref{CubicModelMain}, which has been considered and analyzed before, arises naturally from the separable structure~\eqref{eq:main-structure} and by our smoothness assumptions~\eqref{SmoothnessA},~\eqref{SmoothnessPhi}. Let us discuss an interpretation of Algorithm~\ref{MainAlgorithm} in terms of general problem $\displaystyle \min_{x \in \R^N} F(x)$ with twice-differentiable $F$ (omitting non-differentiable component for simplicity). One can state and minimize the model $m_{H, S}(x; y) \equiv
F(x) + \langle (\nabla F(x))_{[S]}, y \rangle + \frac{1}{2}\langle (\nabla^2 F(x))_{[S]}y, y \rangle + \frac{H}{6}\| y_{[S]} \|^3$ which is a \textit{sketched} version of the originally proposed Cubic Newton method~\cite{nesterov2006cubic}. For alternative sketched variants of Newton-type methods but without cubic regularization see~\cite{pilanci2015randomized}.

The latter model~$m_{H, S}(x; y)$ is the same as $M_{H, S}(x; y)$ when inequality~\eqref{SmoothnessA} from the smoothness assumption for $g$ is exact equality, i.e. when the function $g$ is a quadratic with the Hessian matrix $\nabla^2 g(x) \equiv \mA$. Thus, we may use $m_{H, S}(x; y)$ instead of $M_{H, S}(x; y)$, which is still computationally cheap for small $|S|$. However, this model does not give any convergence guarantees for the general $F$, to the best of our knowledge, unless $S = [n]$. But it can be a workable approach, when the separable structure~\eqref{eq:main-structure} is not provided.

Note also, that Assumption~\ref{AssumptionSmoothnessPhi} about Lipschitz-continuity of the Hessian is not too restrictive. Recent result~\cite{grapiglia2017regularized} shows that Newton-type methods with cubic regularization and with a standard procedure of \textit{adaptive} estimation of $\H_F$  automatically fit the actual level of smoothness of the Hessian without any additional changes in the algorithm.

Moreover, step~\eqref{DefinitionOfStep} of the method as the global minimum of $M_{H, S}(x; y)$ is well-defined and computationally tractable even in nonconvex cases~\cite{nesterov2006cubic}. Thus we can try to apply the method to nonconvex objective as well, but without known theoretical guarantees for $S \not= [n]$.

\section{Empirical Risk Minimization} \label{sec:ERM}


One of the most popular examples of optimization problems in machine learning is \textit{empirical risk minimization} problem, which in many cases can be formulated as follows:
\begin{equation} \label{ERMproblem}
\min_{w \in \R^d} \biggl[ P(w) \equiv \frac{1}{m} \sum_{i = 1}^m \phi_i(b_i^T w) + \lambda g(w) \biggr],
\end{equation} 
where $\phi_i$ are convex \textit{loss functions}, $g$ is a \textit{regularizer}, variables $w$ are \textit{weights} of a model and $m$ is a size of a dataset.

\subsection{Constrained Problem Reformulation} 
Let us consider the case, when the dimension $d$ of problem~\eqref{ERMproblem} is very \textit{huge} and $d \gg m$. This asks us to use some coordinate-randomization technique. Note that formulation~\eqref{ERMproblem} does not directly fit our problem setup~\eqref{eq:main-structure}, but we can easily transform it to the following constrained optimization problem, by introducing new variables $\alpha_i \equiv b_i^T w$:
\begin{equation} \label{ERMproblemConstrained}
\min_{\substack{w \in \R^d\, \\ \alpha \in \R^m} } \biggl[  \frac{1}{m} \sum_{i = 1}^m \phi_i(\alpha_i) + \lambda g(w) + \sum_{i = 1}^m \mathbb{I}\{ \alpha_i = b_i^T w \}  \biggr].
\end{equation}
Following our framework, on every step we will sample a random subset of coordinates $S \subset [d]$ of weights $w$, build the  cubic model of the objective (assuming that $\phi_i$ and $g$ satisfy~\eqref{SmoothnessA},~\eqref{SmoothnessPhi}):
\begin{align*}
&M_{H, S}(w, \alpha; y, h) \equiv \lambda\biggl(  \bigl\langle (\nabla g(w))_{[S]}, y \bigr\rangle + \frac{1}{2}\bigl\langle \mA_{[S]}y, y \bigr\rangle  \biggr) + \\
&+ \frac{1}{m} \biggl(\sum_{i = 1}^m \Bigl( \phi_i^{\prime}(\alpha_i)h_i + \frac{1}{2}\phi_i^{\prime \prime}(\alpha_i)h_i^2 \Bigr) + \frac{H}{6}\|h\|^3 \biggr) + P(w)
\end{align*}
and minimize it by $y$ and $h$ on the affine set:
\begin{equation} \label{ERMproblemConstrainedSubproblem}
(y^{*}, h^{*}) \; := \argmin_{\substack{y \in \R^d_{[S]}, h \in \R^m \\[3pt] \text{subject to} \; h = \mB y } } M_{H, S}(w, \alpha; y, h),
\end{equation}
where rows of matrix $\mB \in \R^{m \times d}$ are $b_i^T$. Then, updates of the variables are: $w^{+} := w + y^{*}$ and $\alpha^{+} := \alpha + h^{*}.$

The following lemma is addressing the issue of how to solve~\eqref{ERMproblemConstrainedSubproblem}, which is required on every step of the method. Its proof can be found in the supplementary material.

\begin{lemma}\label{lem:b897sg99} Denote by $\hat{\mB} \in \R^{m \times |S|}$ the submatrix of $\mB$ with row indices from $S$, by $\hat{\mA} \in \R^{|S| \times |S|}$ the submatrix of $\mA$ with elements whose both indices are from $S$, by $b_1 \in \R^{|S|}$ the subvector of $\nabla g(w)$ with element indices from $S$. Denote vector $b_2 \equiv \bigr(\phi_i^{\prime}(\alpha_i)\bigr)_{i = 1}^m$ and $b \equiv m \lambda b_1 + \hat{\mB}^T b_2$.
	Define the family of matrices $\mZ(\tau): \R_{+} \to \R^{|S| \times |S|}$:
	$$
	\mZ(\tau) \defeq m \lambda \hat{\mA} + \hat{\mB}^T\Bigl( \diag(\phi_i^{\prime \prime} (\alpha_i)) + \frac{H\tau}{2} \mI  \Bigr)\hat{\mB}.
	$$
	Then the solution $(y^{*}, h^{*})$ of~\eqref{ERMproblemConstrainedSubproblem} can be found from the equations:
	$\mZ(\tau^{*}) y^{*}_S = -b,  \; h^{*} = \hat{\mB} y^{*}_S$,
	where $\tau^{*} \geq 0$ satisfies one-dimensional nonlinear equation:
	$\tau^{*} = \| \hat{\mB} (\mZ(\tau^{*}))^{\dagger} b \| $ and $y^{*}_S \in \R^{|S|}$ is the subvector of the solution $y^{*}$ with element indices from $S$. 
\end{lemma}

Thus, after we find the root of nonlinear \textit{one-dimensional} equation, we need to solve $|S| \times |S|$ linear system to compute $y^{*}$. Then,  to find $h^{*}$ we do one matrix-vector multiplication with the matrix of size $m \times |S|$ . Matrix $\mB$ usually has a sparse structure when $m$ is big, which also should be used in effective implementation.

\subsection{Maximization of the Dual Problem}

Another approach to solving optimization problem~\eqref{ERMproblem} is to maximize its Fenchel dual~\cite{rockafellar1997convex}:
\begin{equation} \label{ERMDualProblem}
\max_{\alpha \in \R^m} \biggl[ D(\alpha) \equiv \frac{1}{m}\sum_{i = 1}^m -\phi_i^{*}(-\alpha_i) - \lambda g^{*}\Bigl( \frac{1}{\lambda m} \mB^T \alpha \Bigr) \biggr],
\end{equation}
where $g^{*}$ and $\{ \phi_i^{*} \}$ are \textit{the Fenchel conjugate} functions of $g$ and $\{ \phi_i \}$ respectively, $f^{*}(s) \defeq \sup_x [ \, \langle s, x \rangle - f(x) \, ]$ for arbitrary $f$. It is know~\cite{bertsekas1978local}, that if $\phi_i$ is twice-differentiable in a neighborhood of $y$ and $\nabla^2 \phi_i(y) \succ 0$ in this area, then its Fenchel conjugate $\phi_i^{*}$ is also twice-differentiable in some neighborhood of $s = \nabla \phi_i (y)$ and it holds: $\nabla^2 \phi^{*}_i(s) = ( \nabla^2 \phi_i (y) )^{-1}$. 

Then, in a case of smooth differentiable $g^{*}$ and twice-differentiable $\phi_i^{*}, i \in [m]$ we can apply our framework to~\eqref{ERMDualProblem}, by doing cubic steps in random subsets of the dual variables $\alpha \in \R^m$. The primal $w \in \R^d$ corresponded to particular $\alpha$ can be computed from the stationary equation
$$
w = \nabla g^{*}\biggl( \frac{1}{\lambda m} \mB^T \alpha \biggr),
$$
which holds for solutions of primal~\eqref{ERMproblem} and dual~\eqref{ERMDualProblem} problems in a case of strong duality.

Let us assume that the function $g$ is $1$-strongly convex (which is of course true for $\ell_2$-regularizer $1/2\|w\|_2^{2}$). Then for $G(\alpha) \equiv \lambda g^{*}\Bigl(\frac{1}{\lambda m} \mB^T \alpha \Bigr)$ the uniform bound for the Hessian exists: $\nabla^2 G(\alpha) \preceq \frac{1}{\lambda m^2} \mB^T \mB$. As before we build the randomized cubic model and compute its minimizer (setting $Q \equiv \bigcap_{i = 1}^m \dom\phi_i^{*}$):
\begin{align*}
&M_{H, S}(\alpha, h) \equiv -D(\alpha) + \lambda \left\langle \nabla g^{*}\biggl(\frac{1}{\lambda m} \mB^T \alpha \biggr), h_{[S]} \right\rangle + \\
&+ \frac{1}{2 \lambda m^2} \bigl\| \mB h_{[S]} \|^2  + \frac{1}{m} \sum_{i \in S}\biggl[ -\bigl\langle \nabla \phi_i^{*}(-\alpha_i), h_i \bigr\rangle + \\
&+ \frac{1}{2} \bigl\langle \nabla^2 \phi_i^{*}(-\alpha_i) h_i, h_i \bigr\rangle \biggr] + \frac{H}{6}\|h_{[S]}\|^3; \; S \subset [m],
\end{align*}
\vspace*{-15pt}
\begin{align*}
T_{H, S}(\alpha) &\equiv \argmin_{\substack{h \in \R_{[S]}^m \\[2pt] \text{subject to} \; \alpha+h \in Q}} M_{H, S}(\alpha, h), \\
M_{H, S}^{*}(\alpha) &\equiv M_{H, S}(\alpha, T_{H, S}(\alpha)).
\end{align*}
Because in general we may not know exact Lipschitz-constant for the Hessians, we do an adaptive search for estimating $H$. Resulting primal-dual scheme is presented in Algorithm~\ref{DualAlgorithm}. When a small subset of coordinates $S$ is used, the most expensive operations become: computation of the objective at a current point $D(\alpha)$ and the matrix-vector product $\mB^T\alpha$. Both of them can be significantly optimized by storing already computed values in memory and updating only changed information on every step.
\begin{algorithm}[ht!]
	\caption{Stochastic Dual Cubic Newton Ascent (SDCNA)}
	\label{DualAlgorithm}
	\begin{algorithmic}[1]
		\STATE \textbf{Parameters:} sampling distribution $\Sam$
		\STATE \textbf{Initialization:} choose initial $\alpha^0 \in Q$ and $H_0 > 0$
		\FOR {$k=0,1,2,\dots$}
		\STATE Make a primal update $w^{k} := \nabla g^{*}\bigl( \frac{1}{\lambda m} \mB^T \alpha^{k}\bigr)$
		\STATE Sample $S_k \sim \Sam$ \\[3pt]
		\STATE While $M_{H_k, S_k}^{*}(\alpha^k) > -D(\alpha^k + T_{H_k, S_k}(\alpha^k))$ do \\[4pt]
		\STATE $\qquad \qquad H_k := 1/2 \cdot  H_k$ \\[3pt]
		\STATE Make a dual update $\alpha^{k + 1} := \alpha^k + T_{H_k, S_k}(x^k)$ \\[2pt]
		\STATE Set $H_{k + 1} := 2 \cdot H_k$
		
		\ENDFOR
	\end{algorithmic}
\end{algorithm}

\section{Numerical experiments} \label{sec:numerics}

\subsection{Synthetic} We consider the following synthetic regression task: \[\min_{x \in \R^N} \frac{1}{2}\|\mA x - b\|_2^2 + \sum_{i = 1}^N \frac{c_i}{6}|x_i|^3\] with randomly generated parameters and for different $N$. On each problem of this type we run Algorithm~\ref{MainAlgorithm} and evaluate total computational time until reaching $10^{-12}$ accuracy in function residual. Using middle-size blocks of coordinates on each step is the best choice in terms of total computational time, comparing it with small coordinate subsets and with full-coordinate method. 
This agrees with the provided complexity estimates for the method: an increase of the batch size speeds up convergence rate linearly, but slows down the cost of one iteration cubically. Therefore, the optimum size of the block is on a medium level.

\begin{figure}[ht]
	\begin{center}
		\includegraphics[width=0.32\columnwidth]{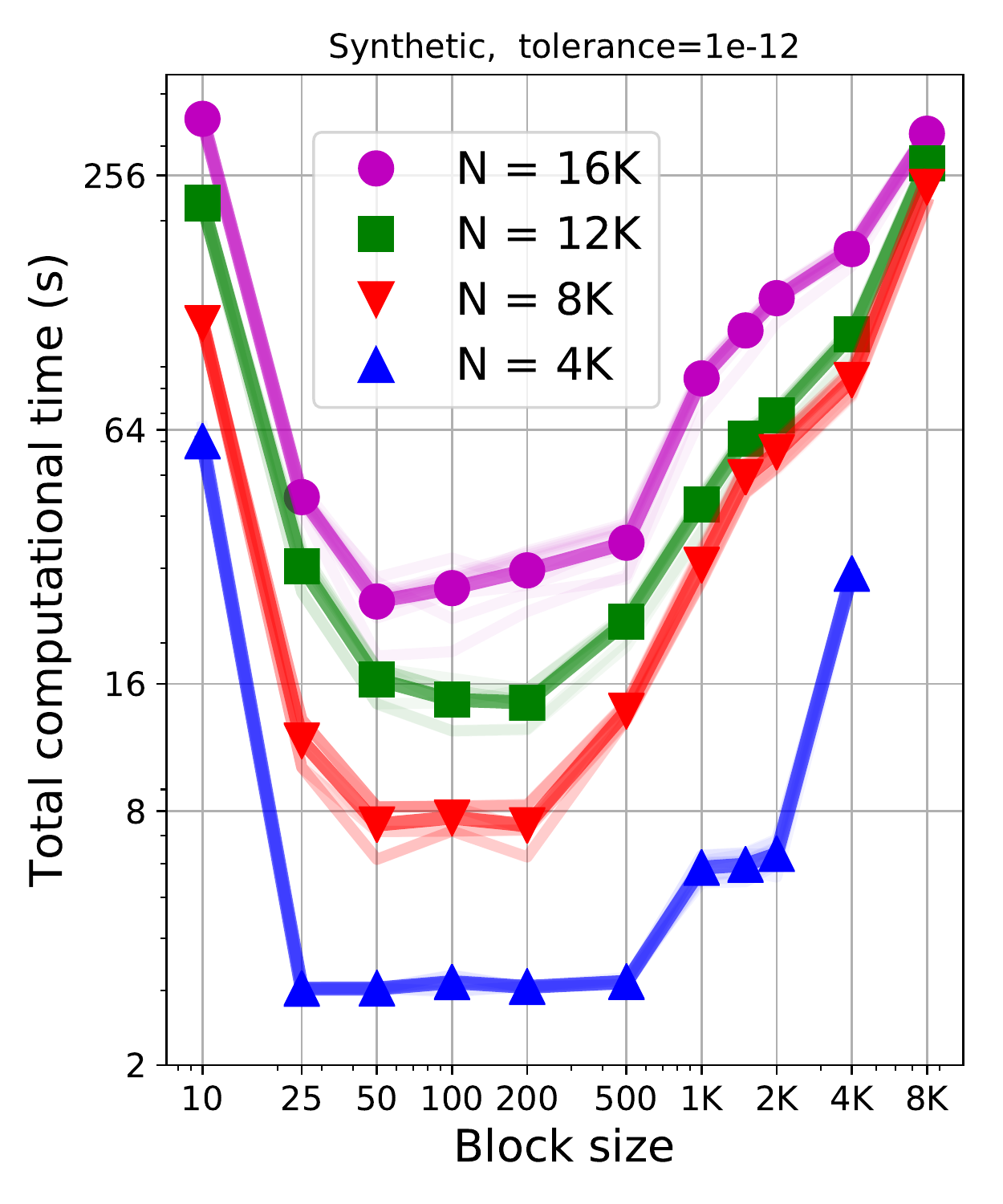}
		\includegraphics[width=0.32\columnwidth]{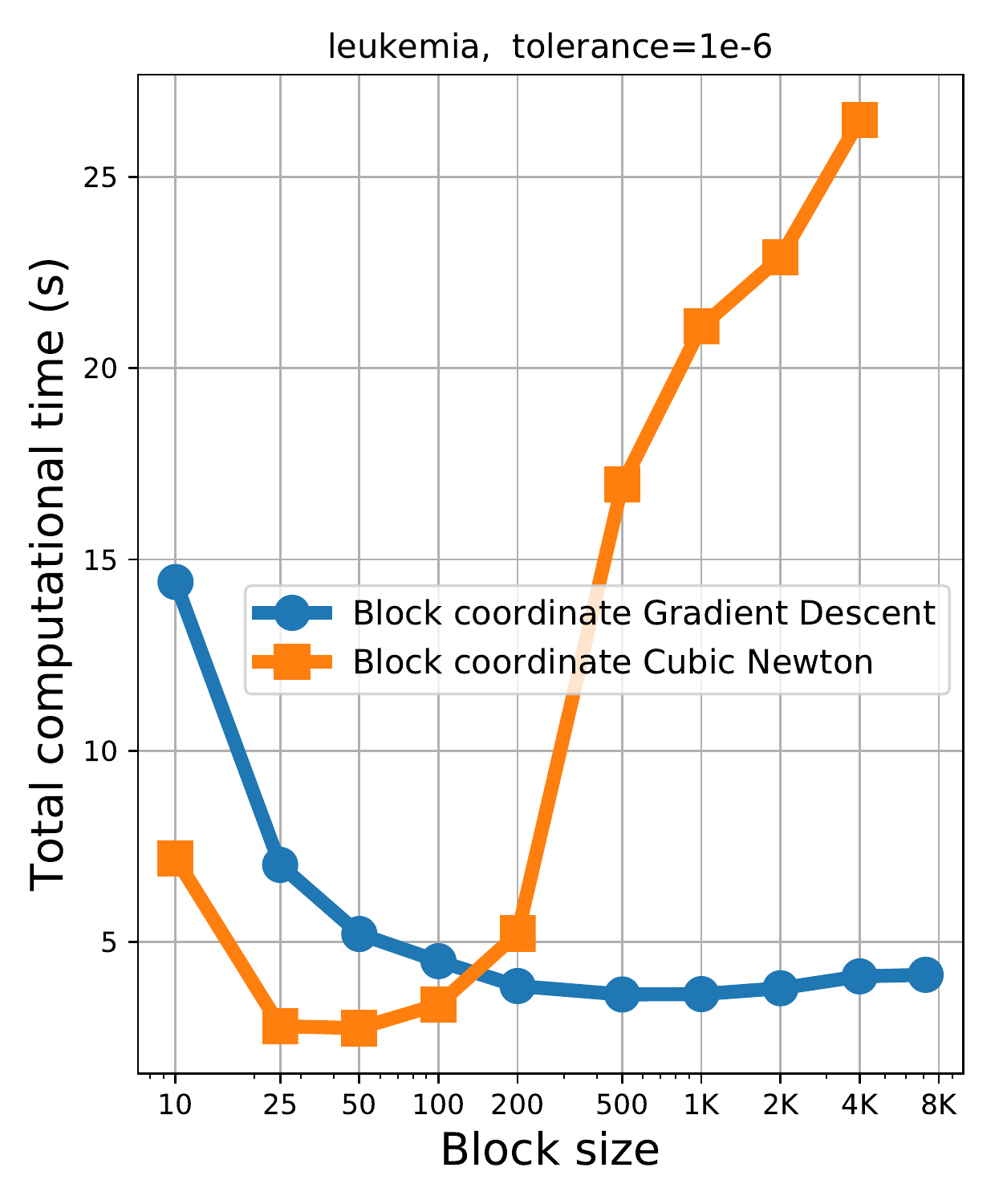}
		\includegraphics[width=0.32\columnwidth]{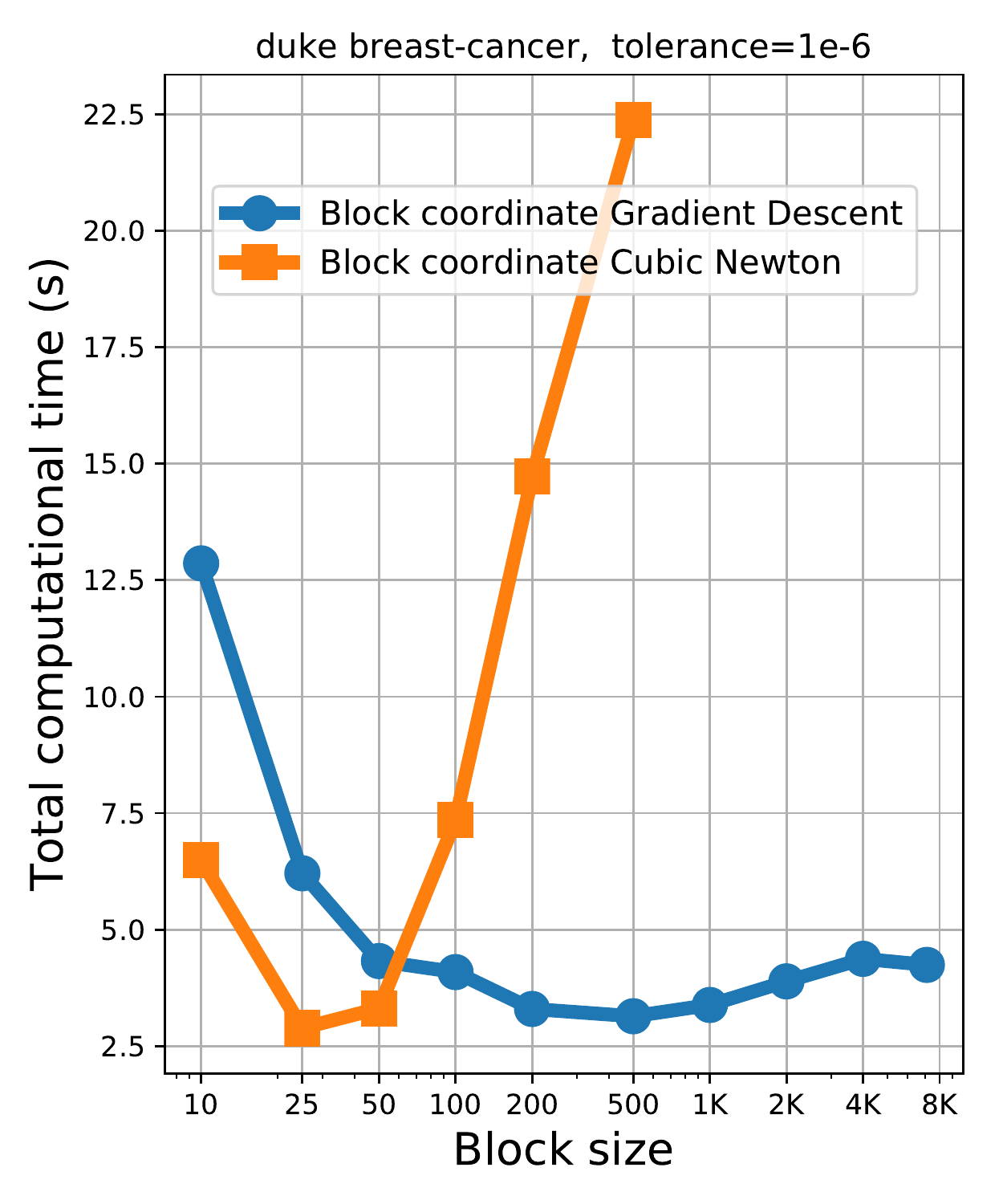}
		\caption{Time it takes to solve a problem for different sampling block sizes. Left: synthetic problem. Center and right: logistic regression for real data. }
		\label{graph1}
	\end{center}
\end{figure}

\subsection{Logistic regression} In this experiment we train $\ell_2$-regularized logistic regression model for classification task with two classes by its constrained reformulation~\eqref{ERMproblemConstrained} and compare the Algorithm~\ref{MainAlgorithm} with the Block coordinate Gradient Descent (see, for example, \cite{tseng2009coordinate}) on the datasets\footnote{\url{http://www.csie.ntu.edu.tw/~cjlin/libsvmtools/datasets/}}: \textit{leukemia}  ($m=38, d=7129$) and \textit{duke breast-cancer} ($m=44, d=7129$). We see that using coordinate blocks of size $25 - 50$ for the Cubic Newton outperforms all other cases of both methods in terms of total computational time. Increasing block size further starts to significantly slow down the method because of high cost of every iteration.

\subsection{Poisson regression} In this experiment we train Poisson model for regression task with integer responses by the primal-dual Algorithm~\ref{DualAlgorithm} and compare it with SDCA~\cite{shalev2013stochastic} and SDNA~\cite{qu2016sdna} methods on synthetic ($m = 1000, d = 200$) and real data\footnote{\url{https://www.kaggle.com/pablomonleon/montreal-bike-lanes}} ($m = 319, d = 20$). 
Our approach requires smaller number of epochs to reach given accuracy, but 
computational efficiency of every step is the same as in SDNA method. 

\begin{figure}[ht]
	\begin{center}
		\includegraphics[width=0.48\columnwidth]{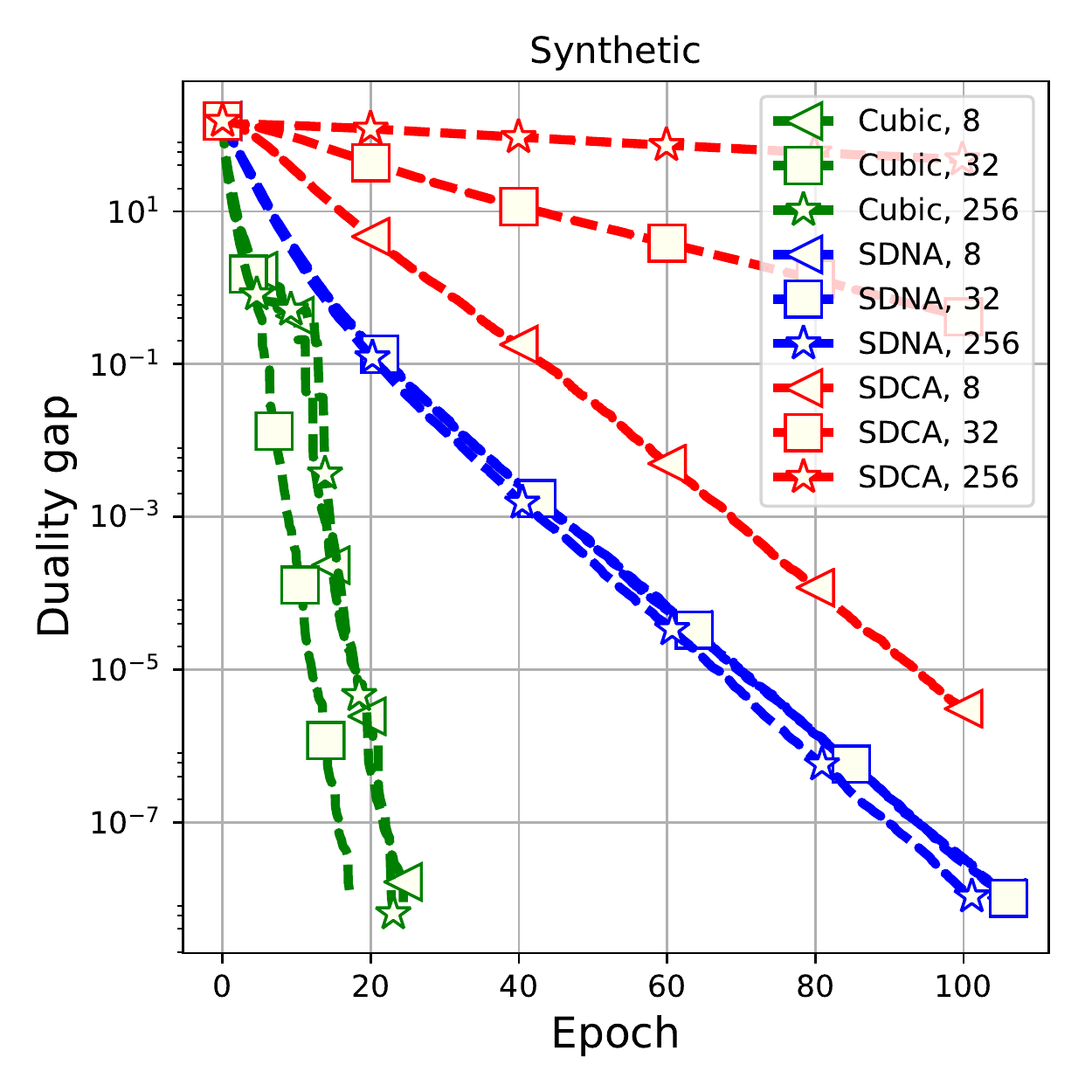}
		\includegraphics[width=0.48\columnwidth]{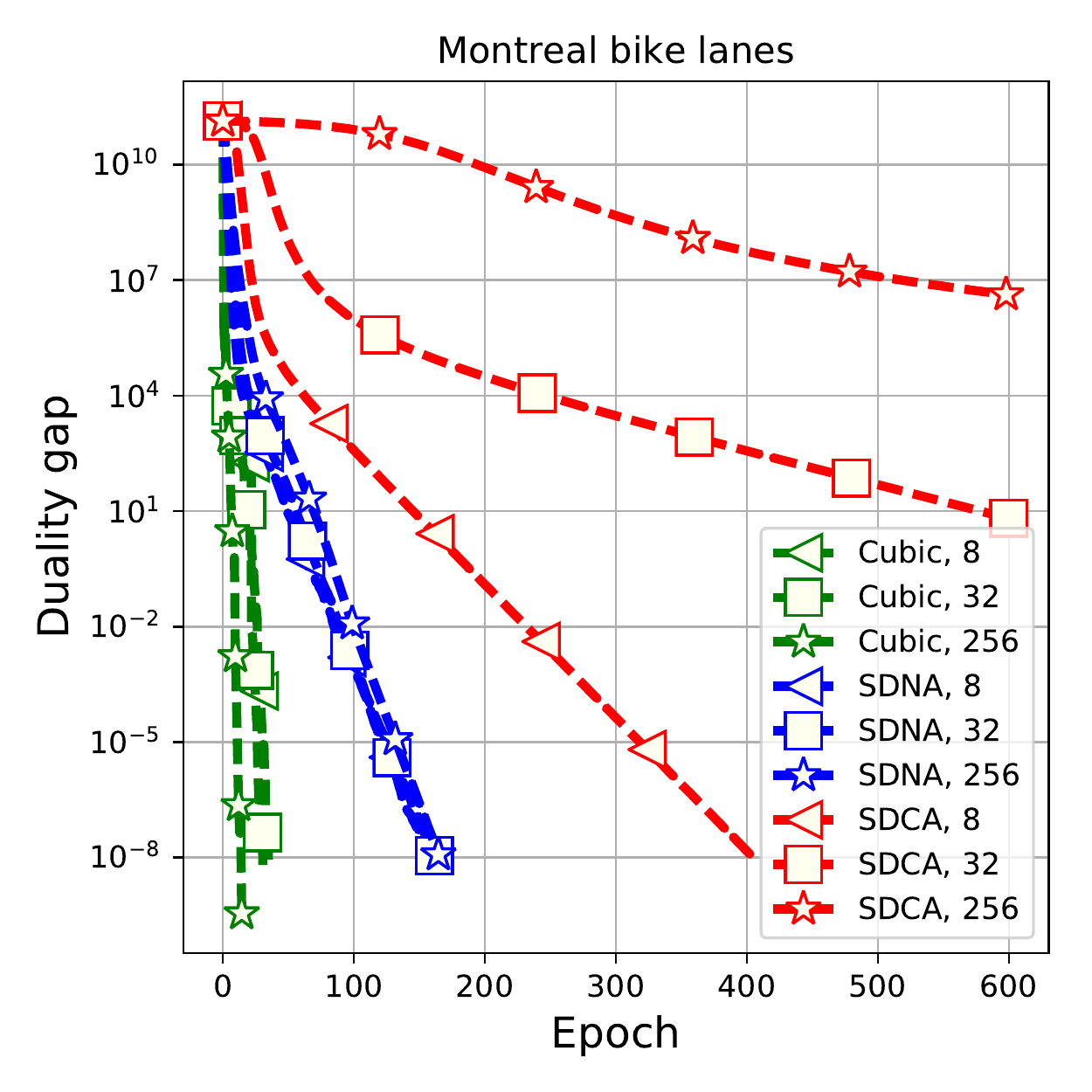}
		\caption{Comparison of Algorithm~\ref{DualAlgorithm} (marked as Cubic) with SDNA and SDCA methods for minibatch sizes $\tau = 8, 32, 256$, training Poisson regression. Left: synthetic. Right: real data.}
		\label{graph2}
	\end{center}
\end{figure}

In the following set of experiments with Poisson regression we take real datasets: 
\textit{australian} ($m = 690, d = 14$), \textit{breast-cancer} ($m = 683, d = 10$), \textit{splice} ($m = 1000, d = 60$), \textit{svmguide3} ($m = 1243, d = 21$) and use for response a random vector $y \in (\mathbb{N}\cup\{0\})^{m}$ from the standard Poisson distribution.

\begin{figure}[ht]
	\begin{center}
		\includegraphics[width=0.48\columnwidth]{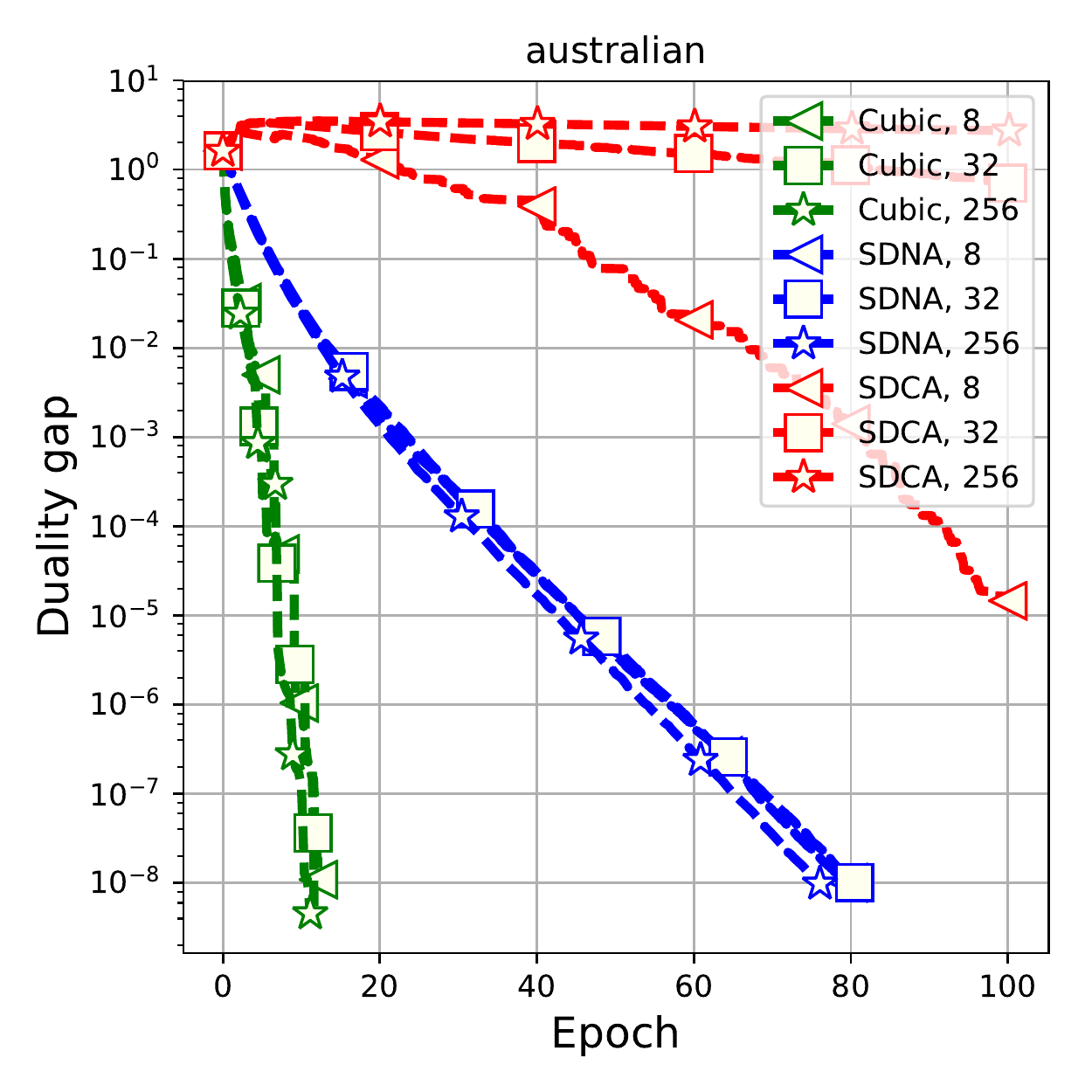}
		\includegraphics[width=0.48\columnwidth]{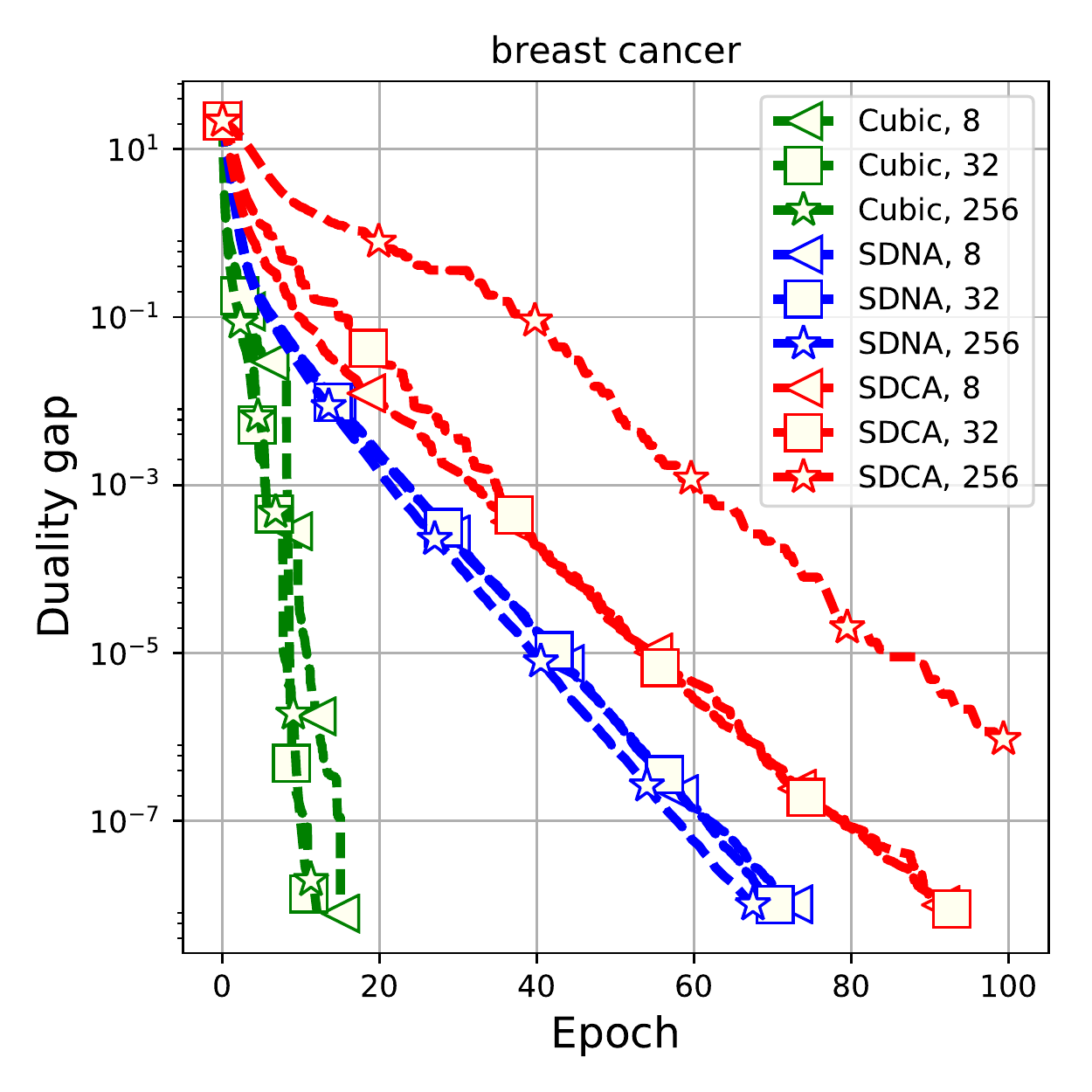}
		\\
		\includegraphics[width=0.48\columnwidth]{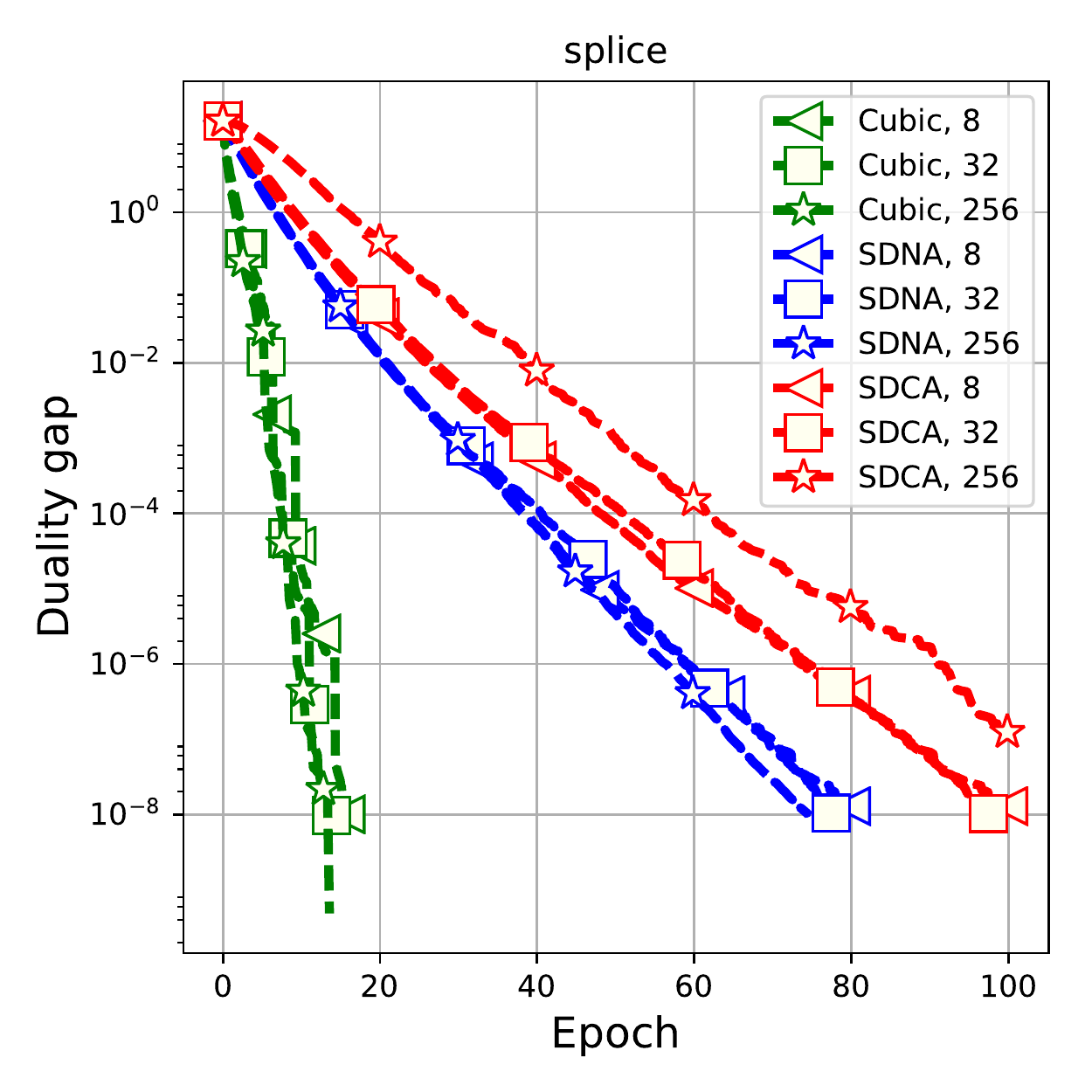}		
		\includegraphics[width=0.48\columnwidth]{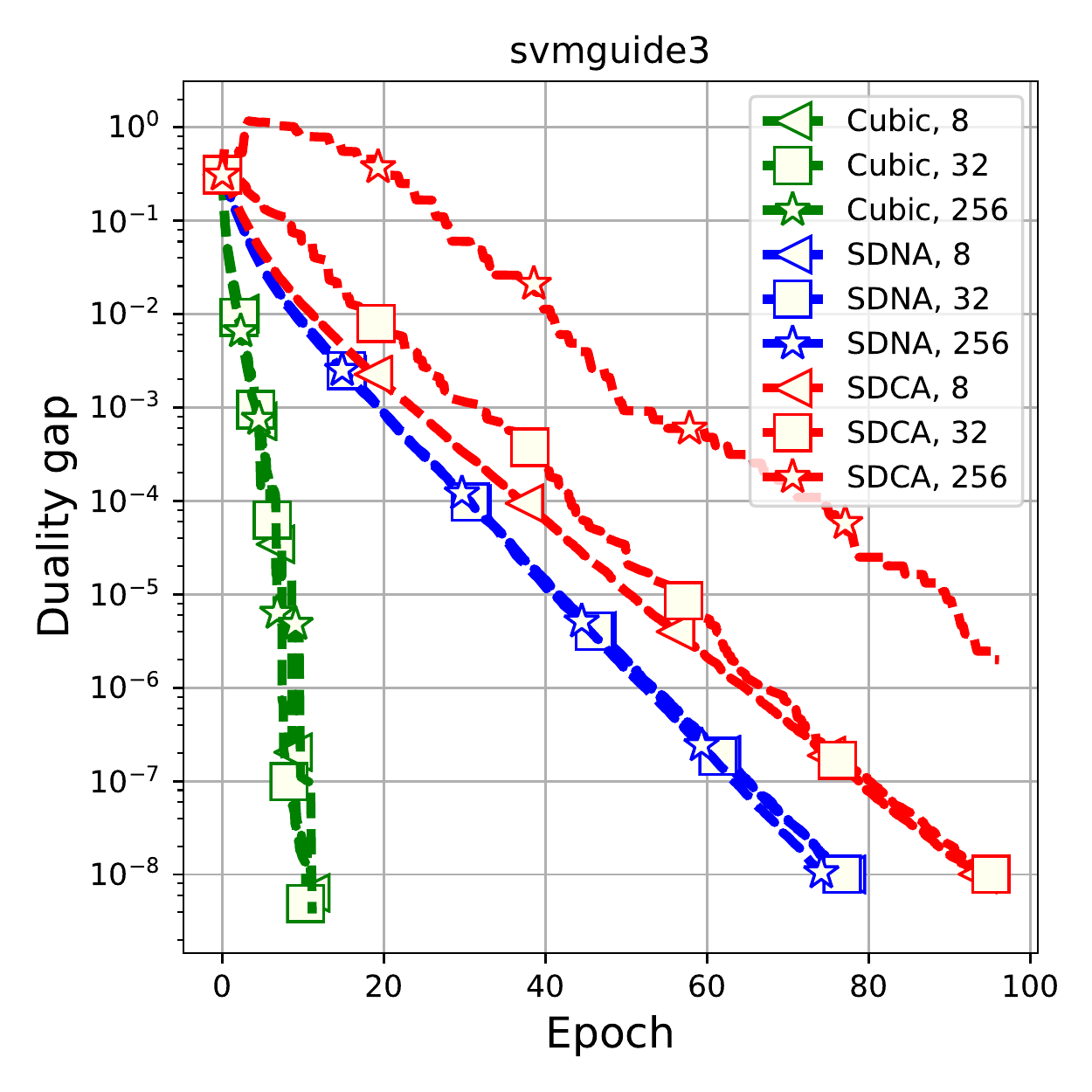}		
		\caption{Comparison of Algorithm~\ref{DualAlgorithm} (marked as Cubic) with SDNA and SDCA methods for minibatch sizes $\tau = 8, 32, 256$, training Poisson regression.}
		\label{graph2}
	\end{center}
	\vskip -0.2in
\end{figure}

We see that in all the cases our method (Algorithm~\ref{DualAlgorithm}) outperforms state-of-the-art analogues in terms of number of data accesses.

\subsection{Experimental setup}

In all our experiments we use $\tau$-nice sampling, i.e., subsets of fixed size chosen from all such subset uniformly at random (note that such subsets overlap). 

\medskip

In the experiments with synthetic cubically regularized regression we generate a data in the following way: sample firstly a matrix $U \in \R^{10 \times N}$, each entry of which is identically distributed from $\mathcal{N}(0, 1)$, then put $A := U^{T} U \in \R^{N \times N}$, $b := -U^T \xi$, where $\xi \in \R^{10}$ is a normal vector: $\xi_i \sim \mathcal{N}(0, 1)$ for each $1 \leq i \leq 10$, and $c_j := 1 + |v_j|$ where $v_j\sim \mathcal{N}(0, 1)$ for all $1 \leq j \leq N$.
Running Algorithm~\ref{MainAlgorithm} we use for $H_k$ the known Lipschitz constants of the Hessians: $H_k := \max_{i \in S_k} c_i$.

\medskip

In the experiments with $\ell_2$-regularized\footnote{Regularization parameter $\lambda$ in all the experiments with logistic and Poisson regression set to $1/m$.} logistic regression we use the \textit{Armijo rule} for computing step length in the Block coordinate gradient descent and the \textit{constant choice} for the parameters $H_k$ in Algorithm~\ref{MainAlgorithm} provided by Proposition~\ref{prop:logistic}.

\medskip

In the experiments with Poisson regression for the methods SDNA and SDCA we use \textit{damped Newton method} as a computational subroutine to compute one method step, using it until reaching $10^{-12}$ accuracy for the norm of the gradient in the corresponding subproblem and the same accuracy for solving inner cubic subproblem in Algorithm~\ref{DualAlgorithm}. For the synthetic experiment with Poisson regression, we generate data matrix $B \in \R^{m \times d}$ as independent samples from standard normal distribution $\mathcal{N}(0, 1)$ and corresponding vector of responses $y \in (\mathbb{N}\cup\{0\})^{m}$ from the standard Poisson distribution (in which \textit{mean} parameter equals to $1$).

\section{Acknowledgments}

The work of the first author was partly supported by Samsung Research, Samsung Electronics, the Russian Science Foundation Grant 17-11-01027 and KAUST.

\clearpage
\bibliography{bibliography}
\bibliographystyle{icml2018}


\clearpage
\onecolumn
\appendix

\icmltitle{Appendix}

\section{Proof of Lemma~\ref{Lemma:LemmaPhiBound}}
\begin{proof}
	Let us show that the Hessian of $\phi_S(x)$ is Lipschitz-continuous with a constant $\displaystyle \max_{i \in S} \{ \H_i \}$. Then~\eqref{LemmaPhiBound} will be fulfilled (see Lemma~1 in~\cite{nesterov2006cubic}). So,
	\begin{align*}
	&\bigl\| \nabla^2 \phi_S(x + h) - \nabla^2 \phi_S(x) \bigr\|^2  \; = \; \max_{\|y\| = 1} \bigl\| \bigl( \nabla^2 \phi_S(x+h) - \nabla^2 \phi_S(x)  \bigr) y \bigr \|^2  \\[7pt]
	& = \; \max_{\|y\| = 1} \sum_{i \in S} \bigl\| \bigl(\nabla^2 \phi_i(x_{(i)} + h_{(i)}) - \nabla^2 \phi_i(x_{(i)}) \bigr) y_{(i)} \bigr \|^2  \\[7pt]
	& \overset{\eqref{SmoothnessPhi}}{\leq} \; \max_{\|y\| = 1} \sum_{i \in S} \Bigl( \H_i^2 \cdot \|h_{(i)}\|^2 \cdot \|y_{(i)}\|^2 \Bigr) \\[7pt]
	& = \; \max_{i \in S} \Bigl\{ \H_i^2 \cdot \|h_{(i)}\|^2  \Bigr\} \; \leq \; \max_{i \in S} \{ \H_i^2 \} \cdot \|h_{[S]}\|^2.
	\end{align*}
\end{proof}

\section{Auxiliary Results}

\begin{lemma} \label{lem:ih88s0} Let Assumptions~\ref{AssumptionStrongConvexityF},~\ref{AssumptionSmoothnessA},~\ref{AssumptionSmoothnessPhi},~\ref{AssumptionUniformSampling} hold. Let minimum of problem~\eqref{eq:main-no-structure} exist, and let $x^{*} \in Q$ be its arbitrary solution: $F(x^{*}) = F^{*}$. 
	
	Then for the random sequence $\{ F(x^k) \}_{k \geq 0}$ generated by Algorithm~\ref{MainAlgorithm} we have a bound: 
	\begin{align} \label{LemmaOneStepProgress}
	\E\bigl[ F(x^{k + 1}) \, | \, x^k  \bigr] \; &\leq \; F(x^{k}) \; - \; \frac{\alpha \tau}{n} \Bigl( F(x^k) - F^{*} \Bigr)  \notag \\[7pt]
	& - \;  \frac{\alpha}{2} \Bigl\langle \Bigl( \frac{\tau}{n} \mG - \alpha\E\bigl[ \mA_{[S_k]} \bigr] \Bigr)(x^k - x^{*}), x^k - x^{*} \Bigr\rangle \; + \; \frac{\alpha^3 \tau}{n} \cdot \frac{\Hf}{2} \cdot \|x^k - x^{*}\|^3,
	\end{align}
	for all $\alpha \in [0, 1]$. 
\end{lemma}

\begin{proof}
	From Assumption~\ref{AssumptionStrongConvexityF} we have the following bound, $\mG \succeq 0$:
	\begin{equation} \label{APPENDIXfIsStronglyConvex}
	\langle \nabla g(x), z \rangle \; \leq \; g(x + z) - g(x) - \frac{1}{2}\langle \mG z, z \rangle, \qquad x, z \in \R^N.
	\end{equation}
	
	\noindent
	Denote $h^k \equiv T_{H_k, S_k}(x^k)$. Then $x^{k + 1} = x^k + h^k$. By a way of choosing $H_k$ we have:
	\begin{align*}
	F(x^{k + 1}) \quad \leq \quad & M^{*}_{H_k, S_k}(x^k) \\
	\quad \equiv \quad& F(x^k) \; + \; \langle (\nabla g(x^k))_{[S_k]}, h^k \rangle \; + \; \frac{1}{2} \langle \mA_{[S_k]} h^k, h^k \rangle \; \\[5pt]
	&+ \;\frac{1}{2} \sum_{i \in S_k}\Bigl( \langle \nabla \phi_i(x_{(i)}^k), h_{(i)}^k\rangle + \nabla^2 \phi_i(x_{(i)}^k) h_{(i)}^k, h_{(i)}^k \rangle \Bigr) \; + \; \frac{H_k}{6} \|h_{[S_k]}^k\|^3 \; \\[5pt]
	&+ \; \sum_{i \in S_k} \Bigl( \psi_i( x_{(i)}^k + h_{(i)}^k) \; - \; \psi_i(x_{(i)}^k) \Bigr).
	\end{align*}
	
	\noindent
	At the same time, because $h^k$ is a minimizer of the cubic model $M_{H_k, S_k}(x^k, y)$, we can change $h^k$ in the previous bound to arbitrary $y \in \R^N$, such that $x^{k} + y \in Q$:
	\begin{align} \label{APPENDIXInitialBound}
	F(x^{k + 1}) \quad \leq \quad  F(x^k) \; +& \; \langle (\nabla g(x^k))_{[S_k]}, y \rangle \; + \; \frac{1}{2}\langle \mA_{[S_k]} y, y \rangle \; + \;\sum_{i \in S_k} \Bigl( \langle \nabla \phi_i(x_{(i)}^k), y_{(i)} \rangle + \frac{1}{2} \langle \nabla^2 \phi_i(x_{(i)}^k) y_{(i)}, y_{(i)} \rangle  \Bigr)  \notag \\[7pt]  
	&\qquad + \quad \frac{H_k}{6}\|y_{[S_k]}\|^3 + \sum_{i \in S_k}\Bigl(\psi_i(x_{(i)}^k + y_{(i)}) - \psi_i(x_{(i)}^k) \Bigr) \notag \\[7pt]
	\leq \quad F(x^k) \; +& \; \langle (\nabla g(x^k))_{[S_k]}, y \rangle \; + \; \frac{1}{2}\langle \mA_{[S_k]} y, y \rangle \; + \; \sum_{i \in S_k} \Bigl( \phi_i(x_{(i)}^k + y_{(i)}) - \phi_i(x_{(i)}^k) \Bigr) \notag \\[7pt]  
	&\qquad + \quad \frac{H_k + \Hf}{6}\|y_{[S_k]}\|^3 \; + \; \sum_{i \in S_k}\Bigl(\psi_i(x_{(i)}^k + y_{(i)}) - \psi_i(x_{(i)}^k) \Bigr),
	\end{align}
	where the last inequality is given by Lemma~\ref{Lemma:LemmaPhiBound}. Note that
	$$
	\E \bigl[ \|y_{[S_k]} \|^3 \bigr] \; \leq \; \E \bigl[ \|y\| \cdot \| y_{[S_k]} \|^2  \bigr] \; = \;  \|y\| \cdot \E\Biggl[ \sum_{i \in S_k} \|y_{(i)}\|^2  \Biggr]  \; = \;  \|y\| \cdot \Biggl( \frac{\tau}{n} \sum_{i = 1}^n \|y_{(i)}\|^2 \Biggr) \; = \; \frac{\tau}{n} \|y\|^3.
	$$
	Thus, by taking conditional expectation $\E\bigl[ \cdot \, | \, x^k \bigr]$ from~\eqref{APPENDIXInitialBound} we have
	\begin{align*}
	\E\bigl[ F(x^{k + 1}) \, | \, x^k \bigr] \quad \leq \quad F(x^k) \; +& \; \frac{\tau}{n} \langle \nabla g(x^k), y \rangle \; + \; \frac{1}{2} \langle \E\bigl[ \mA_{[S_k]} \bigr] y, y \rangle \\[5pt]
	+& \; \frac{\tau}{n} \sum_{i = 1}^n \Bigl( \phi_i(x_{(i)}^k + y_{(i)}) \; + \; \psi_i(x_{(i)}^k + y_{(i)}) \; - \; \phi_i(x_{(i)}^k) \; - \; \psi_i(x_{(i)}^k) \Bigr) \; + \; \frac{\tau}{n} \cdot \frac{\Hf}{2}\|y\|^3, 
	\end{align*}
	which is valid for arbitrary $y \in \R^N$ such that $x^k + y \in Q$. Restricting $y$ to the segment: $y = \alpha (x^{*} - x^k)$, where $\alpha \in [0, 1]$, by convexity of $\phi_i$ and $\psi_i$ we get:
	\begin{align*}
	\E\bigl[ F(x^{k + 1}) \, | \, x^k \bigr] \quad \leq \quad F(x^k) \; +& \; \frac{\alpha \tau}{n} \langle \nabla g(x^k), x^{*} - x^k \rangle \; + \; \frac{\alpha^2}{2} \langle \E\bigl[ \mA_{S_k} \bigr] x^{*} - x^k, x^{*} - x^k \rangle \\[5pt]
	+& \; \frac{\alpha\tau}{n} \sum_{i = 1}^n \Bigl( \phi_i(x^{*}_{(i)}) + \psi_i(x^{*}_{(i)}) - \phi_i(x^k_{(i)}) - \psi_i(x^k_{(i)})\Bigr) \; + \; \frac{\alpha^3\tau}{n} \cdot \frac{\Hf}{2}\|x^{*} - x^k\|^3.
	\end{align*}
	
	\noindent
	Finally, using inequality~\eqref{APPENDIXfIsStronglyConvex} for $x \equiv x^k$ and $z \equiv x^{*} - x^k$ we get the state of the lemma.
\end{proof}

The following technical tool is useful for analyzing a convergence of the random sequence $\xi_k \equiv F(x^k) - F^{*}$ with high probability. It is a generalization of result from~\cite{richtarik2014iteration}.

\begin{lemma} \label{LemmaRandomSequenceConvergence}
	Let $\xi_0 > 0$ be a constant and consider a nonnegative nonincreasing sequence of random variables $\{ \xi_k \}_{k \geq 0}$ with the following property, for all $k \geq 0$:
	\begin{equation} \label{RandomSequenceConvergence}
	\E\bigl[ \xi_{k + 1} \, | \, \xi_k  \bigr] \; \leq \; \xi_k - \frac{\xi_k^p}{c},
	\end{equation}
	where $c > 0$ is a constant and $p \in \{3/2, \, 2\}$.
	Choose confidence level $\rho \in (0, 1)$. 
	
	Then if we set $0 < \varepsilon < \min\{ \xi_0, c^{1 / (p - 1)} \}$ and 
	\begin{equation}
	K \; \geq \; \frac{c}{\varepsilon^{p - 1}} \biggl( \frac{1}{p - 1} + \log \frac{1}{\rho} \biggr) - \frac{c}{\xi_0^{p - 1} (p - 1)},
	\end{equation}
	we have
	\begin{equation}
	\P( \xi_K \leq \varepsilon ) \; \geq \; 1 - \rho.
	\end{equation}
\end{lemma}

\begin{proof}
	Technique of the proof is similar to corresponding one of Theorem~1 from~\cite{richtarik2014iteration}. For a fixed $0 < \varepsilon < \min\{\xi_0, c^{1/(p - 1)} \}$ define a new sequence of random variables $\{ \xi_k^{\varepsilon} \}_{k \geq 0}$ by the following way:
	$$
	\xi_k^{\varepsilon} \; = \; \begin{cases}
	\xi_k, \; &\text{if} \; \xi_k > \varepsilon, \\
	0, \; &\text{otherwise}.
	\end{cases}
	$$
	It satisfies
	$$
	\xi_k^{\varepsilon} \leq \varepsilon \quad \Leftrightarrow \quad \xi_k \leq \varepsilon, \qquad k \geq 0,
	$$
	therefore, by Markov inequality:
	$$
	\P(\xi_k > \varepsilon) \; = \; \P(\xi_k^{\varepsilon} > \varepsilon) \; \leq \; \frac{\E[\xi_k^{\varepsilon}]}{\varepsilon},
	$$
	and hence it suffices to show that
	$$
	\theta_K \leq \rho \varepsilon,
	$$
	where $\theta_k \defeq \E[ \xi_k^{\varepsilon} ]$. From the conditions of the lemma we get
	$$
	\E\bigl[ \xi_{k + 1}^{\varepsilon} \, | \, \xi_k^{\varepsilon} \bigr] \; \leq\;  \xi_k^{\varepsilon} - \frac{(\xi_k^{\varepsilon})^p}{c}, \qquad \E\bigl[ \xi_{k + 1}^{\varepsilon} \, | \, \xi_k^{\varepsilon} \bigr] \; \leq\; \Bigl(1 - \frac{\varepsilon^{p - 1}}{c} \Bigr) \xi_k^{\varepsilon}, \qquad k \geq 0,
	$$
	and by taking expectations and using convexity of $t \mapsto t^p$ for $p > 1$ we obtain
	\begin{align} \label{FirstThetaBound}
	\theta_{k + 1} \quad &\leq \quad \theta_k - \frac{\theta_k^p}{c},  \qquad k \geq 0, \\[5pt]
	\label{SecondThetaBound}
	\theta_{k + 1} \quad &\leq \quad \biggl(1 - \frac{\varepsilon^{p - 1}}{c}\biggr)\theta_k, \qquad k \geq 0.
	\end{align}
	Consider now two cases, whether $p = 2$ or $p = 3/2$, and find a number $k_1$ for which we get $\theta_{k_1} \leq \varepsilon$.
	\begin{enumerate}
		\item $p = 2$, then
		$$
		\frac{1}{\theta_{k + 1}} - \frac{1}{\theta_k} \; = \; \frac{\theta_k - \theta_{k + 1}}{\theta_{k + 1} \theta_k} \; \geq \;  \frac{\theta_{k} - \theta_{k + 1}}{\theta_k^2} \;  \overset{\eqref{FirstThetaBound}}{\geq} \; \frac{1}{c},
		$$
		thus we have $\frac{1}{\theta_k} \geq \frac{1}{\theta_0} + \frac{k}{c} = \frac{1}{\xi_0} + \frac{k}{c}$, and choosing $k_1 \geq \frac{c}{\varepsilon} - \frac{c}{\xi_0}$ we obtain $\theta_{k_1} \leq \varepsilon$.
		
		\item $p = 3/2$, then
		$$
		\frac{1}{\theta_{k + 1}^{1/2}} - \frac{1}{\theta_k^{1/2}} \; = \; \frac{\theta_k^{1/2} - \theta_{k + 1}^{1/2}}{\theta_{k + 1}^{1/2} \theta_k^ {1/2}} \; = \; \frac{\theta_k - \theta_{k + 1}}{(\theta_k^{1/2} + \theta_{k + 1}^{1/2}) \theta_{k + 1}^{1/2} \theta_k^{1/2}} \; \geq \; \frac{\theta_k - \theta_{k + 1}}{2 \theta_k^{3/2}} \; \overset{\eqref{FirstThetaBound}}{\geq} \; \frac{1}{2c},
		$$
		thus we have $\frac{1}{\theta_k^{1/2}} \geq \frac{1}{\theta_0^{1/2}} + \frac{k}{2c}$, and choosing $k_1 \geq \frac{2c}{\varepsilon^{1/2}} - \frac{2c}{\xi_0^{1/2}}$ we get $\theta_{k_1} \leq \varepsilon$.
	\end{enumerate}
	Therefore, for both cases $p \in \{3/2, 2\}$ it is enough to choose
	$$
	k_1 \quad \geq \quad \frac{c}{p - 1} \biggl( \frac{1}{\varepsilon ^ {p - 1}} - \frac{1}{\xi_0^{p - 1}} \biggr),
	$$
	for which we get $\theta_{k_1} \leq \varepsilon$. Finally, letting $k_2 \geq \frac{c}{\varepsilon^{p - 1}} \log \frac{1}{\rho}$ and $K \geq k_1 + k_2$, we have
	$$
	\theta_K \; \leq \; \theta_{k_1 + k_2} \; \overset{\eqref{SecondThetaBound}}{\leq} \; \biggl(1 - \frac{\varepsilon^{p - 1}}{c} \biggr)^{k_2} \theta_{k_1} \; \leq \; \exp\biggl(-\frac{k_2 \varepsilon^{p - 1}}{c}\biggr) \theta_{k_1} \; \leq \; \rho \varepsilon.
	$$
\end{proof}

\section{Proof of Theorem~\ref{TheoremConvexSublinear}}
\begin{proof} From bound~\eqref{LemmaOneStepProgress}, using $\mG \succeq 0$ and $\E\bigl[ \mA_{[S_k]} \bigr] \preceq \frac{\tau L}{n} \mI$ we get, for all $\alpha \in [0, 1]$:
	$$
	\E\bigl[ F(x^{k + 1}) \, | \, x^k \bigr] \; \leq \; F(x^k) - \frac{\alpha \tau}{n}\Bigl( F(x^k) - F^{*} \Bigr) \; + \; \frac{\alpha^2 \tau}{2n} \biggl( L \|x^k - x^{*}\|^2 + \alpha \Hf \|x^k - x^{*}\|^3 \biggr).
	$$
	Thus, for the random sequence $\xi_k \equiv F(x^k) - F^{*}$ we have a bound, for all $\alpha \in [0, 1]$:
	\begin{equation} \label{Theorem:ConvexFloatBound}
	\E\bigl[ \xi_{k + 1} \, |\, \xi_k \bigr] \; \leq \; \xi_k - \frac{\alpha \tau \xi_k}{n} + \frac{\alpha^2 \tau}{2 n} D_0,
	\end{equation}
	where $D_0 \equiv \max\bigl\{ LD^2 + \Hf D^3, \, \xi_0 \bigr\}$. Minimum of the right hand side is attained at $\alpha^{*} = \frac{\xi_k}{D_0} \leq 1$,
	which substituting to~\eqref{Theorem:ConvexFloatBound} gives: $\E[\xi_{k + 1}|\xi_k] \; \leq \; \xi_k - \frac{\tau \xi_k^2}{2n D_0}$. Applying Lemma~\ref{LemmaRandomSequenceConvergence} complete the proof.
\end{proof}

\section{Proof of Theorem~\ref{TheoremStronglyConvexBeta}}
\begin{proof}
	From bound~\eqref{LemmaOneStepProgress}, restricting $\alpha$ to the segment$\, [0, \sigma]$ and using~\eqref{BetaDefinition} we get:
	\begin{equation} \label{Theorem:StronglyConvexFloatBound}
	\E\bigl[ \xi_{k + 1} \, | \, \xi_k \bigr] \; \leq \; \xi_k - \frac{\alpha \tau \xi_k}{n} + \frac{\alpha^3 \tau}{2n} \Hf \|x^k - x^{*}\|^3,
	\end{equation}
	where $\xi_k \equiv F(x^k) - F^{*}$ as before. To get the first complexity bound, we rough the right hand side, denoting $D_0 \equiv \max\bigl\{\Hf D^3, \xi_0 \bigr\}$:
	$$
	\E\bigl[ \xi_{k + 1} \, | \, \xi_k \bigr] \; \leq \; \xi_k - \frac{\alpha \tau \xi_k}{n} + \frac{\alpha^3 \tau}{2n} \frac{  D_0}{\sigma^2},
	$$
	minimum of which is attained at $\alpha^{*} = \sigma \sqrt{\frac{2}{3} \frac{\xi_k}{D_0}} \leq \sigma$. Therefore we obtain 
	$$\E[\xi_{k + 1}|\xi_k] \; \leq \; \xi_k - (2/3)^{3/2} \frac{\tau \sigma  \xi_k^{3/2}}{n D_0},$$
	and applying Lemma~\ref{LemmaRandomSequenceConvergence} complete the proof.
\end{proof}

\section{Proof of Theorem~\ref{TheoremStronglyConvexMu}}
\begin{proof}
	Because of strong convexity we know that $\beta > 0$ and all the conditions of Theorem~\ref{TheoremStronglyConvexBeta} are satisfied. Using inequality $F(x^k) - F^{*} \geq \frac{\mu}{2}\|x^k - x^{*}\|^2$ for~\eqref{Theorem:StronglyConvexFloatBound},  we have for every $\alpha \in [0, \sigma]$:
	$$
	\E\bigl[\xi_{k + 1} \, | \, \xi_k\bigr] \; \leq \; \biggl( 1 - \frac{\alpha \tau}{n} + \frac{\alpha^3 \tau}{n} \frac{\Hf D}{\mu} \biggr) \xi_k.
	$$
	Minimum of the right hand side is attained at $\alpha^{*} = \sigma \min\bigl\{ \sqrt{\frac{\mu}{3\Hf D}}, 1 \bigr\}$, substituting of which and taking total expectation gives a recurrence
	\begin{align*}
	\E\bigl[\xi_{k + 1}\bigr] \; &\leq \; \Biggl( \, 1 - \frac{2\tau \sigma}{3n} \sqrt{\min\Bigl\{ \frac{\mu}{\Hf D}, 1  \Bigr\}} \, \Biggr) \E\bigl[ \xi_{k} \bigr] \; \leq \; \dots \; \leq \; \Biggl( \, 1 - \frac{2\tau \sigma}{3n} \sqrt{\min\Bigl\{ \frac{\mu}{\Hf D}, 1  \Bigr\}} \, \Biggr)^{\!k} \xi_{0}\\[7pt]
	&\leq \; \exp\Biggl( \, - (k + 1)\frac{2\tau \sigma}{3n} \sqrt{\min\Bigl\{ \frac{\mu}{\Hf D}, 1  \Bigr\}} \, \Biggr) \xi_0.
	\end{align*}
	Thus, choosing $K$ large enough, by Markov inequality we have $\P(\xi_K > \varepsilon) \; \leq \; \E[ \xi_K ] \varepsilon^{-1} \; \overset{\eqref{Theorem:StronglyConvexLinearRate}}{\leq} \; \rho.$
\end{proof}

\section{Proof of Lemma~\ref{lem:b897sg99}}

Using notation from the statement of the lemma and multiplying everything by $m$, we can formulate our target optimization subproblem as follows:
\begin{equation} \label{APPENDIXKKTSubproblem}
\min_{\substack{x \in \R^{|S|}, h \in \R^m \\ \text{s.t.} \; h = \hat{\mB}x}} \biggl[ m\lambda \langle b_1, x \rangle \; + \; \frac{m \lambda}{2} \langle \hat{\mA}x, x \rangle \; + \; \langle b_2, h \rangle \; + \; \frac{1}{2} \langle \mD h, h \rangle \; + \; \frac{H}{6}\|h\|^3 \biggr],
\end{equation}
where $\mD \in \R^{m \times m}$ is a diagonal matrix: $\mD \equiv \diag( \phi_i^{\prime \prime}(\alpha_i) )$. We also denote a subvector of $y$ as $x$ to avoid confusion. The minimum of~\eqref{APPENDIXKKTSubproblem} satisfies the following KKT system:
\begin{equation} \label{APPENDIXKKTSystem}
\begin{cases}
&m\lambda b_1 \; + \; m \lambda \hat{\mA}x \; + \; \hat{\mB}^T \mu \; = \; 0, \\[5pt]
&b_2 \; + \; \mD h \; + \; \frac{H}{2}\|h\|h \; - \; \mu \; = \; 0, \\[5pt]
&h \; = \; \hat{\mB}x,
\end{cases}
\end{equation}
where $\mu \in \R^{m}$ is a vector of slack variables. From the second and the third equations we get: 
$$
\mu \; = \; b_2 \; + \; \mD\hat{B}x  \; + \; \frac{H}{2}\|\hat{\mB}x\| \hat{\mB}x,
$$
plugging of which into the first one gives:
$$
\underbrace{\biggl( m \lambda \hat{\mA} \; + \; \hat{\mB}^T\Bigl(  \mD + \frac{H}{2}\|\hat{\mB}x\| \Bigr)\hat{\mB}  \biggr)}_{\equiv \; \mZ(\|\hat{\mB}x\|)} \,  x \; = \;  - \,\underbrace{(m \lambda b_1 + \hat{\mB}^T b_2)}_{\equiv \;b}.
$$
Thus, if we have a solution $\tau^{*} \geq 0$ of the one-dimensional equation: 
$$\tau^{*} \; = \; \| \hat{\mB} ( \mZ(\tau^{*}) )^{\dagger} b \|,$$
then we can set
$$
x^{*} \; := \; -(\mZ(\tau^{*}))^{\dagger} b,  \qquad h^{*} \; := \; \hat{\mB} x^{*}, \qquad \mu^{*} \; := \; b_2 + \mD\hat{\mB}x^{*} + \frac{H}{2}\|\hat{\mB}x^{*}\| \hat{\mB}x^{*}.
$$
It is easy to check that $(x^{*}, h^{*}, \mu^{*})$ are solutions of~\eqref{APPENDIXKKTSystem} and therefore of~\eqref{APPENDIXKKTSubproblem} as well.

\section{Lipschitz Constant of the Hessian of Logistic Loss}

\begin{proposition} \label{prop:logistic}
	Loss function for logistic regression $\phi(t) := \log(1 + \exp(t))$ has Lipschitz-continuous Hessian with constant $\H_{\phi} = 1 / (6 \sqrt{3})$. Thus, it holds, for all $t, s \in \R$:
	\begin{equation} \label{APPENDIXLogisticRegression}
	| \phi^{\prime \prime}(t) - \phi^{\prime \prime}(s) | \quad \leq \quad \H_{\phi}|t - s|.
	\end{equation}
\end{proposition}
\begin{proof}
	\medskip
	
	\noindent
	To prove~\eqref{APPENDIXLogisticRegression} it is enough to show: $| \phi^{\prime \prime \prime}(t) | \leq \H_{\phi},$ for all $t \in \R$. Direct calculations give:
	$$
	\phi^{\prime}(t) \; = \; \frac{1}{1 + \exp(-t)}, \qquad \phi^{\prime \prime}(t) \; = \; \phi^{\prime}(t) \cdot (1 - \phi^{\prime}(t)), \qquad \phi^{\prime\prime\prime}(t) \; = \; \phi^{\prime\prime}(t)\cdot (1 - 2 \phi^{\prime}(t)).
	$$
	
	\noindent
	Let us find extreme values of the function $g(t) := \phi^{\prime\prime\prime}(t)$ for which we have $\displaystyle \lim_{t \to -\infty} g(t) \, = \lim_{t \to +\infty} g(t) = 0$.
	
	\noindent
	Stationary points of $g(t)$ are solutions of the equation
	$$
	g^{\prime}(t^{*}) \; = \; \phi^{(4)}(t^{*}) \; = \;\phi^{\prime\prime}(t^{*}) \cdot \bigl[ (1 - 2\phi^{\prime}(t^{*}))^2 - 2\phi^{\prime}(t^{*}) \cdot (1 - \phi^{\prime}(t^{*}) )  \bigr] \; = \; 0
	$$
	which consequently should satisfy $\phi^{\prime}(t^{*}) = \frac{1}{2} \pm \frac{1}{\sqrt{12}}$ and therefore:
	$$
	g(t^{*}) \; = \; \phi^{\prime \prime \prime}(t^{*}) \; = \; \biggl( \frac{1}{2} + \frac{1}{\sqrt{12}}\biggr) \cdot \biggl( \frac{1}{2} - \frac{1}{\sqrt{12}}\biggr) \cdot \biggl( \pm \frac{1}{\sqrt{3}}\biggr) \; = \; \pm \frac{1}{6 \sqrt{3}},
	$$
	from what we get: $| \phi^{\prime \prime \prime}(t) | \leq 1 / (6 \sqrt{3})$.
\end{proof}

\end{document}